\documentclass[11pt,reqno,a4paper]{amsart}

\usepackage{amsmath,amsfonts,amsthm,mathrsfs,amssymb,cite}
\usepackage[usenames]{color}
\usepackage[dvipsnames]{xcolor}
\usepackage{graphicx}
\usepackage{latexsym}
\usepackage{enumerate}
\usepackage{bm}

\newtheorem{thm}{Theorem}[section]
\newtheorem{cor}{Corollary}[section]
\newtheorem{lem}{Lemma}[section]
\newtheorem{prop}{Proposition}[section]

\theoremstyle{definition}
\newtheorem{defn}{Definition}[section]
\theoremstyle{remark}

\newtheorem{rem}{Remark}[section]
\numberwithin{equation}{section}
\newcommand{\R}{\mathbb{R}}
\newcommand{\C}{\mathbb{C}}

\newcommand{\N}{\mathbb{N}}

\renewcommand{\eqref}[1]{\textnormal{(\ref{#1})}}

\numberwithin{equation}{section}

\title[EM Scattering from A Penetrable Corner]{On electromagnetic scattering from a penetrable corner}

\author{Hongyu Liu}
\address{Department of Mathematics, Hong Kong Baptist University, Kowloon, Hong Kong SAR.}
\email{hongyu.liuip@gmail.com; hongyuliu@hkbu.edu.hk}

\author{Jingni Xiao}
\address{Department of Mathematics, Hong Kong Baptist University, Kowloon, Hong Kong SAR.}
\email{xiaojn@live.com }

\begin{document}

\begin{abstract}

This article is concerned with the time-harmonic electromagnetic (EM) scattering from a generic inhomogeneous medium. It is shown that if there is a right corner on the support of the medium, then it scatters every pair of incident EM fields, excluding a possible class of EM fields which are of very particular forms. That is, for every pair of admissible incident EM fields, the corresponding scattered wave fields associated to the medium scatterer cannot be identically vanishing outside the support of the medium. Indeed, we achieve the corner scattering result by establishing a stronger result, that shows the failure of the analytic extension across the corner of certain EM fields satisfying the so-called interior transmission eigenvalue problem. This extends the relevant study in \cite{BPS14} for the acoustic scattering governed by the Helmholtz equation to the electromagnetic case governed by the Maxwell system. Substantial new challenges arise from the corresponding extension from the scalar PDE to the system of PDEs. Our mathematical arguments combine the analysis for interior transmission eigenvalue problems associated to the Maxwell system; the derivation of novel orthogonality relation for the solutions of Maxwell systems; the construction of complex-geometrical-optics (CGO) solutions for the Maxwell system with new $L^p$-estimates $(p>6)$ on the remainder terms; and the proof of the non-vanishing property for the Laplace transform of vectorial homogeneous harmonic polynomials.

\medskip


\medskip

\end{abstract}




\maketitle

\section{Introduction}\label{sec:Intro}

This paper is concerned with the time-harmonic electromagnetic (EM) scattering described by the Maxwell system as follows. Let $\omega\in\mathbb{R}_+$ denote the frequency of the EM wave propagation. Let $\varepsilon(\mathbf{x}), \mu(\mathbf{x}), \sigma(\mathbf{x})$, $\mathbf{x}\in\mathbb{R}^3$, be all $L^\infty(\mathbb{R}^3)$ functions such that $\varepsilon$ and $\mu$ are positive and $\sigma$ is nonnegative. 
The functions $\varepsilon$, $\mu$ and $\sigma$ signify the EM medium parameters in $\mathbb{R}^3$, and are referred to as the electric permittivity, magnetic permeability and electric conductivity, respectively. It is assumed that there exist positive constants $\varepsilon_0$ and $\mu_0$ such that the inhomogeneous medium
\begin{equation}\label{eq:supp}
\Sigma:=\mathrm{supp }\,(\varepsilon-\varepsilon_0) \cup \mathrm{supp }\,(\mu-\mu_0)\cup \mathrm{supp}(\sigma)
\end{equation}
is bounded. That is, the inhomogeneity of the EM medium is compactly supported, and $\varepsilon_0$ and $\mu_0$ signify the permittivity and permeability of the uniformly homogeneous background space. We shall also write $(\Sigma; \varepsilon, \mu, \sigma)$ to signify the inhomogeneous medium. In what follows, we assume that there exists a bounded Lipschitz domain $D$ with a connected complement $\mathbb{R}^3\backslash\overline{D}$ such that $\Sigma\subset D$. Let $\mathbf{E}^{\mathrm{in}}$ and $\mathbf{H}^{\mathrm{in}}$ be a pair of EM waves that are $\mathbb{C}^3$-valued entire solutions to the following Maxwell equations
\begin{equation}\label{eq:EHin_eqn}
\nabla\wedge  \mathbf{E}^{\mathrm{in}}-\mathtt{i}\omega \mu_0 \mathbf{H}^{\mathrm{in}}=0,\qquad \nabla\wedge \mathbf{H}^{\mathrm{in}}+\mathtt{i}\omega\varepsilon_0 \mathbf{E}^{\mathrm{in}}=0\quad\mbox{in}\ \ \mathbb{R}^3,
\end{equation}
where $\mathtt{i}:=\sqrt{-1}$ is the imaginary unit. In the setup of our study, one sends a pair of incident wave fields $(\mathbf{E}^{\mathrm{in}}, \mathbf{H}^{\mathrm{in}})$ to {interrogate} the inhomogeneous medium $(\Sigma; \varepsilon,\mu, \sigma)$, which is a common means in non-destructive wave probe. The inhomogeneity of the medium perturbs the incident waves and produces the so-called scattered EM fields. Let $(\mathbf{E}^{\mathrm{sc}},\mathbf{H}^{\mathrm{sc}})$ and $(\mathbf{E},\mathbf{H}):=(\mathbf{E}^{\mathrm{in}}, \mathbf{H}^{\mathrm{in}})+(\mathbf{E}^{\mathrm{sc}},\mathbf{H}^{\mathrm{sc}})$, respectively, denote the scattered and the total EM fields. Then the EM scattering is described by the following Maxwell system
\begin{equation}\label{eq:Maxwell}
\begin{cases}
& \nabla\wedge\mathbf{E}(\mathbf{x})-\mathtt{i}\omega\mu(\mathbf{x}) \mathbf{H}(\mathbf{x})=0,\hspace*{1.8cm} \mathbf{x}\in\mathbb{R}^3,\medskip\\
&\nabla\wedge\mathbf{H}(\mathbf{x})+(\mathtt{i}\omega\varepsilon(\mathbf{x})-\sigma(\mathbf{x})) \mathbf{E}(\mathbf{x})=0,\quad \mathbf{x}\in\mathbb{R}^3,\medskip\\
& \displaystyle{\lim_{|\mathbf{x}|\rightarrow+\infty}\left(\mathbf{H}^{\mathrm{sc}}(\mathbf{x})\wedge \mathbf{x}-|\mathbf{x}| \mathbf{E}^{\mathrm{sc}}(\mathbf{x})\right)=0}.
\end{cases}
\end{equation}
The last limit in \eqref{eq:Maxwell} is known as the {\it Silver-M\"uller} radiation condition and it holds uniformly in all directions $\hat{\mathbf{x}}:=\mathbf{x}/|\mathbf{x}|\in\mathbb{S}^2$. We refer to \cite{LRX16,Ned01} for the unique existence of $(\mathbf{E},\mathbf{H})\in H_{\mathrm{loc}}(\mathrm{curl}; \mathbb{R}^3)^3$ to the Maxwell system \eqref{eq:Maxwell}. Particularly, one has the following asymptotics as $|\mathbf{x}|\rightarrow+\infty$ (cf. \cite{CoK12}),
\begin{equation*}
\begin{split}
\mathbf{E}(\mathbf{x})=& \mathbf{E}^{\mathrm{in}}(\mathbf{x})+\frac{e^{\mathtt{i}\omega|\mathbf{x}|}}{|\mathbf{x}|} \mathbf{A}_\infty^{\mathbf{E}}(\hat{\mathbf{x}}; \mathbf{E}^{\mathrm{in}}, \mathbf{H}^{\mathrm{in}})+\mathcal{O}\left(\frac{1}{|\mathbf{x}|^2}\right), \\
\mathbf{H}(\mathbf{x})=& \mathbf{H}^{\mathrm{in}}(\mathbf{x})+\frac{e^{\mathtt{i}\omega|\mathbf{x}|}}{|\mathbf{x}|} \mathbf{A}_\infty^{\mathbf{H}}(\hat{\mathbf{x}}; \mathbf{E}^{\mathrm{in}}, \mathbf{H}^{\mathrm{in}})+\mathcal{O}\left(\frac{1}{|\mathbf{x}|^2}\right).
\end{split}
\end{equation*}
There is the following one-to-one correspondence,
\begin{equation*}
\mathbf{A}_\infty^{\mathbf{H}}(\hat{\mathbf{x}}; \mathbf{E}^{\mathrm{in}}, \mathbf{H}^{\mathrm{in}})=\hat{\mathbf{x}}\wedge \mathbf{A}_\infty^{\mathbf{E}}(\hat{\mathbf{x}}; \mathbf{E}^{\mathrm{in}}, \mathbf{H}^{\mathrm{in}}).
\end{equation*}
In what follows, we set $\mathbf{A}_\infty(\hat{\mathbf{x}})$ to signify either $\mathbf{A}_\infty^{\mathbf{E}}(\hat{\mathbf{x}})$ or $\mathbf{A}_\infty^{\mathbf{H}}(\hat{\mathbf{x}})$, and it is known as the far-field pattern or the scattering amplitude. An important inverse scattering problem arising from scientific and technological applications is to recover $(\Sigma;\varepsilon,\mu,\sigma)$ by knowledge of $\mathbf{A}_\infty(\hat{\mathbf{x}})$ \cite{COS09,CoK12,Isa06,OPS93,OlS96,Uhl13}.

In this paper, we are particularly interested in the scenario that one has $\mathbf{A}_\infty(\hat{\mathbf{x}})\equiv 0$, which is closely related to a significant engineering application, the so-called {\it invisibility cloaking} (cf. \cite{GKL09,GKL09cloaking,Uhl09}). By Rellich's Theorem (cf. \cite{CoK12}), if $\mathbf{A}_\infty(\hat{\mathbf{x}}; \mathbf{E}^{\mathrm{in}}, \mathbf{H}^{\mathrm{in}})\equiv 0$, then one has $(\mathbf{E}^{\mathrm{sc}}(\mathbf{x}), \mathbf{H}^{\mathrm{sc}}(\mathbf{x}))\equiv 0$ for $\mathbf{x}\in\mathbb{R}^3\backslash\overline{D}$. That is, in such a case, by sending the interrogating wave fields $(\mathbf{E}^{\mathrm{in}}, \mathbf{H}^{\mathrm{in}})$, one observes no perturbation of the EM wave propagation outside the target/scattering object $(D; \varepsilon,\mu,\sigma)$, and hence the scatterer is invisible to the outside observer.
The invisibility cloaking has received significant attention in the scientific community in recent years due to its practical importance. Blueprints for achieving invisibility via the use of the artificially engineered {\it metamaterials} were proposed in \cite{GLU03,Leo06,PSS06}. Materials therein are anisotropic and singular. It is of scientific curiosity and practical importance to know whether one can achieve the invisibility by regular and isotropic materials. As an important consequence of our results in the present paper, we affirm this question with a negative answer when there is a right corner on the cloaking device. Indeed, we shall show that for a generic isotropic inhomogeneous EM medium, if there is a right corner on the support of the medium, then it scatters every pair of incident EM fields, excluding a possible class of EM fields of special forms. That is, for every pair of admissible incident EM fields, the corresponding scattered wave fields associated to the medium scatterer cannot be identically vanishing outside the inhomogeneous medium. The corner scattering result is an immediate consequence of a stronger result that shall be established in the current article. In fact, we show the failure of the analytic extension across the corner of certain electromagnetic fields satisfying the so-called interior transmission eigenvalue problem in that corner.

This work extends the relevant study in \cite{BPS14} for the acoustic scattering governed by the Helmholtz equation. It is proved in \cite{BPS14} that if the support of a generic inhomogeneous acoustic medium has a $90^\circ$ corner in $\R^n$, then it scatters every incident acoustic wave field. It is further shown in \cite{PSV14arXiv} that under similar conditions, if the support of the acoustic medium has a conical corner (with the exception of a discrete set of opening angles in 3D under which nothing is known so far) in $\R^2$ or $\R^3$, then it scatters every incident acoustic wave field. Authors in \cite{ElH15} applied different arguments to show the similar corner scattering results mentioned above in $\R^2$, as well as some edge scattering result in $\R^3$. Using the corner scattering result in \cite{BPS14}, the authors in \cite{HSV16} considered the inverse problem on recovering the shape of an inhomogeneous acoustic medium supported in a polyhedral domain by a single far-field pattern. The aforementioned acoustic results have been quantified by providing sharp stability estimates in a recent paper \cite{BlL16prepare}. The current paper is the first one to deal with the corner scattering for the vectorial electromagnetic waves governed by the Maxwell system. Substantial new challenges arise from the corresponding extension from scalar Helmholtz equation to the system of Maxwell equations.

Our main results on corner scattering associated with EM waves are stated in Section 2. The proofs of the main results are given in Section 3. Our mathematical arguments combine the analysis for interior transmission eigenvalue problems associated to the Maxwell system; the derivation of a novel orthogonality relation for the solutions of Maxwell systems; the construction of complex-geometrical-optics (CGO) solutions for the Maxwell system with new $L^p$-estimates $(p>6)$ on the remainder terms; and the proof of the non-vanishing property for the Laplace transform of vectorial homogeneous harmonic polynomials.
The structure of our proof is sketched as follows. The identically vanishing of the far-field pattern can leads to an interior transmission eigenvalue problem associated to the Maxwell system and this shall be discussed in Section 2. Based on analysis of the interior transmission eigenvalue problem, one can derive a certain orthogonality relation for the solutions of Maxwell systems, and this will be given in Section 3. As trial functions in the orthogonality identity, we shall construct CGO solutions for the Maxwell system with new $L^p$-estimates $(p>6)$ on the remainder terms. This shall be crucial in our study. There are several existing results on the CGO solutions for the Maxwell system \cite{COS09,OPS93,OlS96} using ideas pioneered in \cite{SyU87}, but only associated with $L^2$-estimates which do not fulfil our needs. The new CGO results are stated in Theorem~\ref{thm:thm2.1}, and its proof is given in Section 4. Finally, in order to establish a contradiction based on the orthogonality relation, we shall need a certain non-vanishing property of the Laplace transform of vectorial homogeneous harmonic polynomials, and the corresponding result is stated in Theorem~\ref{thm:LapTrans} and proved in Section 5.

\section{Electromagnetic scattering from a penetrable corner}

\subsection{An interior transmission eigenvalue problem}

Throughout the rest of the paper, for notational simplification, we set
\begin{equation*}
\gamma:=\varepsilon+\mathtt{i}\frac{\sigma}{\omega},
\end{equation*}
and
\begin{equation}\label{eq:k}
k:=\omega\sqrt{\mu_0\varepsilon_0}.
\end{equation}
Consider the EM scattering problem \eqref{eq:supp}--\eqref{eq:Maxwell} and assume that $\mathbf{A}_\infty(\hat{\mathbf{x}}; \mathbf{E}^{\mathrm{in}},$ $\mathbf{H}^{\mathrm{in}})\equiv 0$. By the Rellich theorem, we know that $(\mathbf{E}^{\mathrm{sc}}(\mathbf{x}), \mathbf{H}^{\mathrm{sc}}(\mathbf{x}))\equiv 0$ for $\mathbf{x}\in \mathbb{R}^3\backslash\overline{D}$. Then, by setting $(\mathbf{E}^0, \mathbf{H}^0):=(\mathbf{E}^{\mathrm{in}}, \mathbf{H}^{\mathrm{in}})$ and $(\mathbf{E}^{-}, \mathbf{H}^{-}):=(\mathbf{E}, \mathbf{H})$ in \eqref{eq:supp}--\eqref{eq:Maxwell}, one can show that there holds
\begin{equation}\label{eq:ite1}
\begin{cases}
\nabla\wedge \mathbf{E}^{-}-\mathtt{i}\omega\mu \mathbf{H}^{-}=0,\quad
\nabla\wedge \mathbf{H}^{-}+\mathtt{i}\omega\gamma \mathbf{E}^{-}=0,
&\mbox{in $D$},\\
\nabla\wedge \mathbf{E}^0-\mathtt{i}\omega\mu_0 \mathbf{H}^0=0,\quad
\nabla\wedge \mathbf{H}^0+\mathtt{i}\omega\varepsilon_0 \mathbf{E}^0=0,
&\mbox{in $D$},\\
\qquad\mathbf{n}\wedge \mathbf{E}^{-}=\mathbf{n}\wedge \mathbf{E}^0,\quad
\qquad\mathbf{n}\wedge \mathbf{H}^{-}=\mathbf{n}\wedge \mathbf{H}^0,
&\mbox{on $\partial D$},
\end{cases}
\end{equation}
where $\mathbf{n}\in\mathbb{S}^2$ is the exterior unit normal vector to $\partial D$. The system \eqref{eq:ite1} is known as the interior transmission eigenvalue problem. If there exist nontrivial pairs $(\mathbf{E}^{-}, \mathbf{H}^{-})$ and $(\mathbf{E}^0, \mathbf{H}^0)$ fulfilling \eqref{eq:ite1}, then they are referred to as the interior transmission eigenfunctions associated with the interior transmission eigenvalue $\omega$. That is, if there is no scattering for \eqref{eq:supp}--\eqref{eq:Maxwell}, then the total wave fields $(\mathbf{E},\mathbf{H})$ and the incident wave fields $(\mathbf{E}^{\mathrm{in}}, \mathbf{H}^{\mathrm{in}})$ form the interior transmission eigenfunctions to \eqref{eq:ite1}. On the other hand, if the interior transmission eigenfunctions $(\mathbf{E}^0, \mathbf{H}^0)$ to \eqref{eq:ite1} can be extended to be a pair of entire solutions to the Maxwell system \eqref{eq:EHin_eqn}, then using the extended functions as the incident fields $(\mathbf{E}^{\mathrm{in}}, \mathbf{H}^{\mathrm{in}})$, there would be no scattering for \eqref{eq:supp}--\eqref{eq:Maxwell}. However, we shall show that if there is a right corner on the support of the inhomogeneous medium $(D; \varepsilon, \mu, \sigma)$, then the aforementioned extension cannot hold true, unless $(\mathbf{E}^0, \mathbf{H}^0)$ are of a very particular form. This in turn implies that if there is a right corner on the support of the EM medium, then it scatters every pair of incident EM fields, excluding a possible class of EM fields of special forms. Finally, we would like to refer to \cite{Kir07,CaK10} for the relevant study on the existence of interior transmission eigenvalues and eigenfunctions for the problem \eqref{eq:ite1}.

\subsection{Statement of the main results}

Before stating the main results, we briefly introduce some preliminary knowledge on real-analytic functions. The next lemma is a well-known result for solutions to \eqref{eq:EHin_eqn} (cf. \cite{CoK12}).

\begin{lem}\label{lem:analy}
Suppose that a non-trivial pair $(\mathbf{E}^0,\mathbf{H}^0)$ satisfies \eqref{eq:EHin_eqn} in a neighborhood of a certain point $\mathbf{x}_0\in \mathbb{R}^3$. Then $\mathbf{E}^0$ and $\mathbf{H}^0$ are (real) analytic in this neighborhood.
\end{lem}

In what follows, we let $\bm{\alpha}=(\alpha^{(1)},\alpha^{(2)},\alpha^{(3)})\in \mathbb{N}_0^3$ with $\mathbb{N}_0:=\mathbb{N}\cup\{0\}$ denote a multi-index, and $|\bm{\alpha}|:=\sum_{j=1}^3\alpha^{(j)}$. For $\mathbf{x}=(x^{(j)})_{j=1}^3$, we define
\begin{equation}\label{eq:indexMulti}
\mathbf{x}^{\bm{\alpha}}:=\prod_{j=1}^3(x^{(j)})^{\alpha^{(j)}}.
\end{equation}
It is recalled that for any real-analytic function, say $f(\mathbf{x})$, around the point $\mathbf{x}_0=\mathbf{0}\in \mathbb{R}^3$, one has the following Taylor series expansion in a neighborhood of $\mathbf{x}_0$,
\begin{equation}\label{eq:tse}
f(\mathbf{x})=\sum_{\bm{\alpha}\in\mathbb{N}_0^3, |\bm{\alpha}|\geq 0} C_{\bm{\alpha}} \mathbf{x}^{\bm{\alpha}},
\end{equation}
where $C_{\bm{\alpha}}$ are complex-valued constants depending on $\mathbf{x}_0$ and ${\bm{\alpha}}$. Let $N_0\in \mathbb{N}_0$ be the integer such that  $C_{\bm{\alpha}}=0$ in \eqref{eq:tse}, for any multi-index $\bm{\alpha}$ with $|\bm{\alpha}|<N_0$, and $C_{\bm{\alpha}_0}\neq 0$ for some $|\bm{\alpha}_0|=N_0$. Then
\[
P_{N_0}(\mathbf{x})=\sum_{\bm{\alpha}\in \mathbb{N}_0^3, |\bm{\alpha}|=N_0} C_{\bm{\alpha}} \mathbf{x}^{\bm{\alpha}}
\]
is said to be the lowest-order homogeneous polynomial of the Taylor series at $\mathbf{x}_0=\mathbf{0}$ for $f$, with the lowest-order $N_0$.
Given a vector field $\mathbf{P}$, we say that $\mathbf{P}$ is a \emph{homogeneous polynomial of order $N>0$}, if $\mathbf{P}$ is not identically zero, and each Cartesian component of $\mathbf{P}$ is either identically zero, or a homogeneous polynomial of order $N$, $j=1,2,3$.
The vector field $\mathbf{P}$ is called harmonic if all of its Cartesian components are harmonic.

\begin{defn}\label{defn:defn2.1}
	Let $\mathbf{V}=(V^{(j)})_{j=1}^3$ be a 3-dimensional analytic function in a neighborhood of $\mathbf{x}_0=\mathbf{0}\in\mathbb{R}^3$.
	Let ${P}^{(j)}$ be the lowest-order homogeneous polynomial of the Taylor series at $\mathbf{0}$ for $V^{(j)}$, and let $N_j$ be the order of ${P}^{(j)}$, $j=1,2,3$. Denote
	$$N=\min_{1\leq j\leq 3} N_j.$$
	Set
	\begin{equation*}
	P_{N}^{(j)}=\left\{
	\begin{array}{cc}
	0&\mbox{if $N_j>N$},\medskip\\
	{P}^{(j)}&\mbox{if $N_j=N$},
	\end{array}\right. \quad j=1,2,3.
	\end{equation*}	
	Then $\mathbf{P}_{N}=\mathbf{P}_{N}[\mathbf{V}]:=(P_N^{(j)})_{j=1}^3$ is called the \emph{lowest-order homogeneous polynomial} of the Taylor series of $\mathbf{V}$ at $\mathbf{x}_0=\mathbf{0}$ with the \emph{lowest order} $N_\mathbf{V}:=N$.
\end{defn}

\begin{defn}\label{defn:admiType}
A pair of real-analytic functions $(\mathbf{E}^0,\mathbf{H}^0)$ in a neighborhood of $\mathbf{x}_0=\mathbf{0}$ is said to be \emph{inadmissible} if it fulfils the conditions described in what follows.
	
Let $\mathbf{P}[\mathbf{E}^0]$ (resp. $\mathbf{P}[\mathbf{H}^0]$) be the lowest-order homogeneous polynomial of the Taylor series of $\mathbf{E}^0$ (resp. $\mathbf{H}^0$) at $\mathbf{x}_0=\mathbf{0}$, with the lowest order $N_{\mathbf{E}^0}$ (resp. ${N}_{\mathbf{H}^0}$). Set $N:=\min\{N_{\mathbf{E}^0}, N_{\mathbf{H}^0}\}$, and
\begin{equation}\label{eq:notS}
\mathcal{S}:=\begin{cases}
\{\mathbf{E}^0\}\quad & \mbox{if}\ \ N=N_{\mathbf{E}^0}<N_{\mathbf{H}^0},\\
\{\mathbf{H}^0\}\quad & \mbox{if}\ \ N=N_{\mathbf{H}^0}<N_{\mathbf{E}^0},\\
\{\mathbf{E}^0, \mathbf{H}^0\}\quad & \mbox{if}\ \ N=N_{\mathbf{E}^0}=N_{\mathbf{H}^0}.
\end{cases}
\end{equation}
One has $N\ge 1$ and moreover,
	\begin{enumerate}
		\item[(I)] if $N$ is odd, then there has for all $\mathbf{S}\in \mathcal{S}$ that
		\begin{equation}\label{eq:Podd}
		\mathbf{P}[\mathbf{S}]=\left( x^{(j)} \, P_{N-1}^{(j)}[\mathbf{S}] (\mathbf{x})\right)_{j=1}^3 ,
		\end{equation}
		where $P_{N-1}^{(j)}[\mathbf{S}]$, $j=1,2,3$, are homogeneous polynomials of order $N-1$;
		\item[(II)] if $N$ is even, then for all $\mathbf{S}\in \mathcal{S}$,
		\begin{equation}\label{eq:Peven}
		\mathbf{P}[\mathbf{S}]=\left( P^{(j)}_{N-2}[\mathbf{S}](\mathbf{x})\, \prod_{\substack{l=1\\l\neq j}}^{3} x^{(l)}\right) _{j=1}^{3},
		\end{equation}
		where $P_{N-2}^{(j)}[\mathbf{S}]$, $j=1,2,3$, are homogeneous polynomials of order $N-2$.
	\end{enumerate}
\end{defn}

\begin{defn}\label{defn:admiType1}
	A single real-analytic function $\mathbf{E}^0$ is called \emph{inadmissible} if the pair $(\mathbf{E}^0,\mathbf{H}^0)$, with $\mathbf{H}^0:=\mathbf{E}^0$, is inadmissible. A pair of or a single real-analytic function(s) is referred to as \emph{admissible} if it is not inadmissible.
\end{defn}

We are now in a position to state the main results on corner scattering. Henceforth, we denote by $\mathcal{K}$ the positive orthant in $\mathbb{R}^3$, namely
\begin{equation}\label{eq:orth1}
\mathcal{K}:=\{\mathbf{x}=(x^{(j)})^3_{j=1}\in\mathbb{R}^3; x^{(j)}>0,j=1,2,3\}.
\end{equation}

\begin{thm}\label{thm:main}
Consider an EM medium with parameters $\gamma$ and $\mu$ as described in Section~\ref{sec:Intro}, and assume that $\gamma$ and $\mu$, possibly after a rigid change of coordinates, can be represented as
		\begin{equation}\label{eq:EMR1}
		\gamma=\phi_{\gamma}\chi_{\overline{\mathcal{K}}}+\varepsilon_0,\quad \mu=\phi_{\mu}\chi_{\overline{\mathcal{K}}}+\mu_0,
		\end{equation}
		where $\phi_{\gamma},\phi_{\mu}\in{C_c^{3}(\mathbb{R}^3)}$
		are such that
		\begin{equation}\label{eq:EMR2}
		\phi_{\gamma}(\mathbf{0})\neq 0,\quad \phi_{\mu}(\mathbf{0})\neq 0.
		\end{equation}
		Moreover, there exists a bounded Lipschitz domain $D$ in $\R^3$ with a connected complement $\mathbb{R}^3\backslash\overline{D}$ such that 
		\begin{equation}
		 D\supset\mathrm{supp }\,({\gamma}-\varepsilon_0) \cup \mathrm{supp }\,({\mu}-\mu_0)
		\end{equation}
		and that there exists a neighborhood $\mathcal{N}_{\epsilon}$ of $\mathbf{0}$ satisfying
		\begin{equation}
		\overline{D\cap \mathcal{N}_{\epsilon}}=\overline{\mathcal{K}\cap \mathcal{N}_{\epsilon}}=:{\mathcal{N}}^+_{\epsilon}.
		\end{equation}
		
	Suppose that there exist nontrivial pairs $(\mathbf{E}^0,\mathbf{H}^0)$ and $(\mathbf{E}^{-}, \mathbf{H}^{-})$ satisfying \eqref{eq:ite1} in $D$, and $(\mathbf{E}^0, \mathbf{H}^0)$ is admissible; that is, $(\mathbf{E}^0, \mathbf{H}^0)$ is not of the particular type described in Definition~\ref{defn:admiType}. Then $\mathbf{E}^0$ and $\mathbf{H}^0$ cannot be simultaneously extended into any neighborhood of $\mathbf{0}$, as solutions to \eqref{eq:EHin_eqn}. 	
\end{thm}

We postpone the proof of Theorem~\ref{thm:main} to Section 3. In Theorem~\ref{thm:main}, according to \eqref{eq:EMR1} and \eqref{eq:EMR2}, at the vertex of the corner $\mathcal{K}$, there are both jump discontinuities for the EM parameters $\gamma$ and $\mu$. However, this is not essential and indeed, the following result can also be obtained and its proof shall be given in Section 3 as well.

\begin{thm}\label{thm:thm1.2}
Suppose that $\gamma$ is the same as described in Theorem~\ref{thm:main}, and $\mu\equiv\mu_0$.
Let $(\mathbf{E}^0,\mathbf{H}^0)$ and $(\mathbf{E}^{-}, \mathbf{H}^{-})$ be non-trivial solutions satisfying \eqref{eq:ite1} in $D$, and $\mathbf{E}^0$ is admissible as defined in Definition~\ref{defn:admiType1}.
Then $\mathbf{E}^0$ cannot be extended into any neighborhood of $\mathbf{0}$ as an electric solution to \eqref{eq:EHin_eqn}.
	
Analogously, if $\gamma$ is identically $\varepsilon_0$ while $\mu$ is the same as described in Theorem~\ref{thm:main}, and $\mathbf{H}^0$ is not inadmissible. Then $\mathbf{H}^0$ cannot be extended into any neighborhood of $\mathbf{0}$ as a magnetic solution to \eqref{eq:EHin_eqn}.
\end{thm}

As an immediate and important consequence of Theorems~\ref{thm:main} and \ref{thm:thm1.2}, we have
\begin{cor}\label{cor:1}
Consider an EM scatterer with parameters $(\gamma,\mu)$ satisfying conditions in Theorem~\ref{thm:main} or those in Theorem~\ref{thm:thm1.2}. Then it scatters every pair of incident fields $(\mathbf{E}^{\mathrm{in}}, \mathbf{H}^{\mathrm{in}})$ that are not of the particular forms described in Definition~\ref{defn:admiType}.
\end{cor}

It is required in Theorems~\ref{thm:main} and \ref{thm:thm1.2} that $(\mathbf{E}^0, \mathbf{H}^0)$ is admissible, that is, $\mathbf{E}^0$ and $\mathbf{H}^0$ are not of the particular forms described in Definition~\ref{defn:admiType}. However, we would like to emphasize that this might be a technicality condition mainly due to our mathematical arguments. In other words, it is unclear to us whether if $(\mathbf{E}^0, \mathbf{H}^0)$ is inadmissible, then both of them or one of them can or cannot be extended to be entire solution(s) to the Maxwell system \eqref{eq:EHin_eqn} to form a pair of incident waves. The same remark equally holds for Corollary~\ref{cor:1}. This point is definitely worth of future investigation. Moreover, for plane waves and {point waves} which are widely used in the EM scattering theory, both of them are admissible in our theorems as remarked in the following. 

{
\begin{rem}
		It is recalled that one needs to require that $N\ge 1$ in Definitions~\ref{defn:admiType} and \ref{defn:admiType1} for the inadmissible waves. However, for any plane incident waves of the form
		\begin{equation}\label{eq:incPlane}
		\mathbf{E}^{\mathrm{in}}= e^{\mathtt{i}k\mathbf{x}\cdot\mathbf{d}}\mathbf{d}^\perp, \quad
		\mathbf{H}^{\mathrm{in}}= \sqrt{\varepsilon_0/\mu_0}e^{\mathtt{i}k\mathbf{x}\cdot\mathbf{d}}\mathbf{d}\wedge\mathbf{d}^\perp,
		\end{equation}
		with $k:=\omega\sqrt{\varepsilon_0\mu_0}$, and $\mathbf{d},\mathbf{d}^\perp\in\mathbb{S}^2$ being perpendicular to each other, one can easily show that $N=0$ in this case and hence they do not belong to the inadmissible class. Moreover, by virtue of Theorem~\ref{thm:admiEquiv} in what follows, where further characterizations of inadmissible waves are given, one can show that the following point EM waves,
		\begin{equation}\label{eq:incPoint}
		\mathbf{E}^{\mathrm{in}}=\nabla\wedge\left( \mathbf{a}\Phi_{\mathbf{y}}\right) , \quad \mathbf{H}^{\mathrm{in}}= \frac{1}{\mathtt{i}\omega\mu_0}\nabla\wedge \nabla\wedge\left( \mathbf{a}\Phi_{\mathbf{y}}\right),
		\end{equation}
		with $\mathbf{y}\notin \mathcal{N}_{\epsilon}$ and
		\begin{equation}
		\Phi_{\mathbf{y}}(\mathbf{x}):=\frac{e^{\mathtt{i}k|\mathbf{x}-\mathbf{y}|}}{4\pi |\mathbf{x}-\mathbf{y}|},\quad \mathbf{x}\neq\mathbf{y},
		\end{equation}
		do not belong to the inadmissible class. 
		That is, if there is a right corner on the support of the inhomogeneous EM medium, then it scatters any incident fields being plane waves or point waves of the forms \eqref{eq:incPlane} and \eqref{eq:incPoint}, respectively. 
			\end{rem}
	}

\subsection{Further characterization of inadmissibility}

In Theorems~\ref{thm:main} and \ref{thm:thm1.2}, we need to require that $\mathbf{E}^0$ and $\mathbf{H}^0$ are not of the particular form described in Definition~\ref{defn:admiType}, that is, they are admissible. Noting that $\mathbf{E}^0$ and $\mathbf{H}^0$ are solutions to the Maxwell system \eqref{eq:EHin_eqn}, we are able to provide further characterization of the inadmissibility of $(\mathbf{E}^0, \mathbf{H}^0)$ in terms of the vectorial spherical harmonic expansions. First, we recall some preliminary knowledge on spherical harmonics by following the presentations in \cite{CoK12,Ned01}.

For integers $l\ge 0$ and $|m|\le l$, let $Y_{l}^{m}$ be the Laplace's spherical harmonic given in the spherical coordinates $(r,\theta,\varphi)$ by
\begin{equation}
Y_{l}^{m}(\hat{\mathbf{x}})=Y_{l}^{m}(\theta,\varphi) = \sqrt{\frac{2l+1}{4\pi}\frac{(l-|m|)!}{(l+|m|)!}} P_l^{|m|}(\cos{\theta}) e^{\mathtt{i}m\varphi},
\end{equation}
where $P_l^{|m|}$ is the associated Legendre polynomial.
Denote by $\nabla_{\mathrm{S}}$ the surface gradient on the sphere $\mathbb{S}^2$.
Define
\begin{equation}\label{eq:Tlm}
\mathbf{T}_{l}^{m}(\hat{\mathbf{x}})=\frac{1}{\sqrt{l(l+1)}}\nabla_{\mathrm{S}}Y_{l}^{m}(\hat{\mathbf{x}})\wedge\hat{\mathbf{x}}, \quad l\ge 1,\, |m|\le l,
\end{equation}
\begin{equation}\label{eq:Ilm}
\mathbf{I}_{l}^{m}(\hat{\mathbf{x}})=\frac{1}{\sqrt{(l+1)(2l+3)}}
\left( \nabla_{\mathrm{S}}Y_{l+1}^{m}(\hat{\mathbf{x}})+(l+1)Y_{l+1}^{m}(\hat{\mathbf{x}})\hat{\mathbf{x}}\right), \quad l\ge 0,\, |m|\le l+1,
\end{equation}
and
\begin{equation}
\mathbf{N}_{l}^{m}(\hat{\mathbf{x}})=\frac{1}{\sqrt{l(2l-1)}}
\left(- \nabla_{\mathrm{S}}Y_{l-1}^{m}(\hat{\mathbf{x}})+lY_{l-1}^{m}(\hat{\mathbf{x}})\hat{\mathbf{x}}\right), \quad l\ge 1,\, |m|\le l-1.
\end{equation}

\begin{lem}[\cite{CoK12,Ned01}]
	The family $\{Y_l^m;\, l\ge 0,|m|\le l \}$ forms an orthonormal basis of $L^2(\mathbb{S}^2)$, and $\{\mathbf{T}_{l}^{m}, \mathbf{I}_{l}^{m}, \mathbf{N}_{l}^{m};\, l,m\}$ is that of $L^2(\mathbb{S}^2)^3$.
\end{lem}

For integers $l\ge 0$ and $m$ with $|m|\le l+1$, define the functions
\begin{equation}\label{eq:Elm}
\mathbf{E}_{l,m}(\mathbf{x}):=j_{l+1}(k|\mathbf{x}|) \mathbf{T}_{l+1}^{m}(\hat{\mathbf{x}}),
\end{equation}
and
\begin{equation}\label{eq:Hlm}
\begin{split}
\mathbf{H}_{l,m}(\mathbf{x})
:=&-\frac{\mathtt{i}}{\omega\mu_0}\nabla\wedge \mathbf{E}_{l,m}(\mathbf{x})\\
=&-\mathtt{i}\frac{\sqrt{\varepsilon_0/\mu_0}}{\sqrt{2l+3}}
\left( \sqrt{l+2}\,j_{l}(k|\mathbf{x}|) \mathbf{I}_{l}^{m}(\hat{\mathbf{x}})
+\sqrt{l+1}\,j_{l+2}(k|\mathbf{x}|) \mathbf{N}_{l+2}^{m}(\hat{\mathbf{x}})\right),
\end{split}
\end{equation}
where $k$ is the number defined in \eqref{eq:k},
and $j_l$, $l\ge 0$, are the spherical Bessel functions given by
\begin{equation}\label{eq:jl}
j_l(t):=\sum_{n=0}^{\infty}(-1)^n \frac{ (l+n)!\, 2^l}{ n!\,(2l+2n+1)! }\,t^{l+2n},\quad t\in\R.
\end{equation}

\begin{lem}[\cite{Ned01}]\label{lem:lem2.300}
	Any pair of functions $(\mathbf{E}^0,\mathbf{H}^0)$ that satisfies \eqref{eq:EHin_eqn} in a neighborhood $\mathcal{N}_\epsilon$ of $\mathbf{x}_0=\mathbf{0}$ can be expanded in $\mathcal{N}_\epsilon$ as a linear combination of the pairs of functions in $$\{(\mathbf{E}_{l,m},\mathbf{H}_{l,m}),\, (-\mu_0/\varepsilon_0 \mathbf{H}_{l,m},\mathbf{E}_{l,m});\, l\ge 0, |m|\le l+1\},$$
	where for each $l\ge 0$ and $|m|\le l+1$, the functions $\mathbf{E}_{l,m}$ and $\mathbf{H}_{l,m}$ are given by \eqref{eq:Elm} and \eqref{eq:Hlm}, respectively.	
\end{lem}

Next we introduce the following two $6$-dimensional fields,
\begin{equation}
(\mathbf{E}\mathbf{H})_{l,m}=(\mathbf{E}_{l,m},\mathbf{H}_{l,m})\quad\text{and}\quad (\mathbf{H}\mathbf{E})_{l,m}=(-\mu_0/\varepsilon_0 \mathbf{H}_{l,m},\mathbf{E}_{l,m}),
\end{equation}
for all integers $l\ge 0$ and $|m|\le l+1$. Then, we have

\begin{thm}\label{thm:admiEquiv}
	A non-trivial pair $(\mathbf{E}^0,\mathbf{H}^0)$ satisfies \eqref{eq:EHin_eqn} in a neighborhood $\mathcal{N}_\epsilon$ of $\mathbf{x}_0=\mathbf{0}\in \mathbb{R}^3$ is of the inadmissible type in Definition~\ref{defn:admiType} if and only if they can be represented in $\mathcal{N}_\epsilon$ as,
	\begin{equation}\label{eq:EHinReq1}
	\mathbf{E}^0=\mathbf{E}_0+\mathbf{E}_{\mathbf{H}}+\tilde{\mathbf{E}},\quad \mathbf{H}^0=\mathbf{H}_0+\mathbf{H}_{\mathbf{E}}+\tilde{\mathbf{H}},
	\end{equation}
	where
	\begin{equation}
	(\mathbf{E}_0,\mathbf{H}_{\mathbf{E}})=\sum_{
		m=(l_0+1)\, \mathrm{mod}\, 2
	}^{[(l_0+1)/2]}a_{l_0,m} \left( (\mathbf{E}\mathbf{H})_{l_0,2m} + (-1)^{l_0}(\mathbf{E}\mathbf{H})_{l_0,-2m}\right),
	\end{equation}
	\begin{equation}
	(\mathbf{E}_{\mathbf{H}},\mathbf{H}_0)=\sum_{
		m=(l_0+1)\, \mathrm{mod}\, 2
	}^{[(l_0+1)/2]}b_{l_0,m} \left( (\mathbf{H}\mathbf{E})_{l_0,2m} + (-1)^{l_0}(\mathbf{H}\mathbf{E})_{l_0,-2m}\right),
	\end{equation}
	and
	\begin{equation}
	(\tilde{\mathbf{E}},\tilde{\mathbf{H}})=\sum_{l>l_0}\sum_{|m|\le l+1}
	\left( a_{l,m}(\mathbf{E}\mathbf{H})_{l,m} + b_{l,m}(\mathbf{H}\mathbf{E})_{l,m}\right),
	\end{equation}
	where $l_0\ge 1$ and
	\begin{equation}\label{eq:EHinReq5}
	\sum_{	m=(l_0+1)\, \mathrm{mod}\, 2}^{[(l_0+1)/2]}
	\left( a^2_{l_0,m}+b^2_{l_0,m}\right) \neq 0.
	\end{equation}
\end{thm}
The proof of Theorem~\ref{thm:admiEquiv} is a bit lengthy with tedious calculations, and in order to focus on the corner scattering study, we postpone the proof to Section 6.

\section{Proofs of Theorems~\ref{thm:main} and \ref{thm:thm1.2}}\label{sec:2}

\subsection{Auxiliary results}

We first derive an orthogonality identify for solutions to the Maxwell systems as follows.

\begin{lem}
	Let $(\mathbf{E}^0,\mathbf{H}^0)$ and $(\mathbf{E}^{-}, \mathbf{H}^{-})$ solve \eqref{eq:ite1}.
	Then the orthogonality relation
	\begin{equation}\label{eq:ortho}
	\int_{D}(\mu-\mu_0)\mathbf{H}^0\cdot \mathbf{H}-(\gamma-\varepsilon_0)\mathbf{E}^0\cdot \mathbf{E}=0
	\end{equation}
	holds for any $(\mathbf{E},\mathbf{H})$ satisfying
	\begin{align}\label{eq:MaxwellEH_full}
	\nabla\wedge \mathbf{E}-\mathtt{i}\omega \mu \mathbf{H}=0,\qquad&
	\nabla\wedge \mathbf{H}+\mathtt{i}\omega\gamma \mathbf{E}=0 & \mbox{in $D$}.
	\end{align}
\end{lem}
\begin{proof}
	It is noticed by \eqref{eq:MaxwellEH_full} that
	\begin{equation}\label{eq:eq3.3000}
	\mathtt{i}\omega\int_{D}\mu \mathbf{H}^{0}\cdot \mathbf{H}-\gamma \mathbf{E}^{0}\cdot \mathbf{E}
	{=}
	\int_{D} \mathbf{H}^0\cdot \left( \nabla\wedge \mathbf{E}\right) + \mathbf{E}^0\cdot \left(\nabla\wedge \mathbf{H} \right) ,
	\end{equation}
	and by the Maxwell equations satisfied by $\mathbf{E}^0$ and $\mathbf{H}^0$  that
	\begin{equation}\label{eq:eq3.400}
	\mathtt{i}\omega\int_{D}\mu_0 \mathbf{H}^0\cdot \mathbf{H}-\varepsilon_0 \mathbf{E}^0\cdot \mathbf{E}
	{=}
	\int_{D} \mathbf{H}\cdot \left( \nabla\wedge  \mathbf{E}^0\right) + \mathbf{E}\cdot \left(\nabla\wedge  \mathbf{H}^0 \right) .
	\end{equation}
	Hence, subtracting \eqref{eq:eq3.400} from \eqref{eq:eq3.3000}, integrating by parts and then applying the boundary condition in \eqref{eq:ite1}, one can obtain that
	\begin{equation}\label{eq:eq3.500}
	\begin{split}
	&\mathtt{i}\omega\int_{D}(\mu-\mu_0)\mathbf{H}^0\cdot \mathbf{H}-(\gamma-\varepsilon_0)\mathbf{E}^0\cdot \mathbf{E}\\
	&=\int_{\partial D} \mathbf{n} \cdot \left(\mathbf{H}\wedge \mathbf{E}^{0}-\mathbf{E}\wedge \mathbf{H}^{0}\right)\\
	&{=}\int_{\partial D} \mathbf{n} \cdot \left(\mathbf{H}\wedge \mathbf{E}^{-}-\mathbf{E}\wedge \mathbf{H}^{-}\right)
	=0.
	\end{split}
	\end{equation}
	The last equality in \eqref{eq:eq3.500} owes to integration by part in combination with the fact that $(\mathbf{E},\mathbf{H})$ and $(\mathbf{E}^{-}, \mathbf{H}^{-})$ both satisfy the Maxwell equations \eqref{eq:MaxwellEH_full} in $D$.
	
	The proof is complete.
\end{proof}

In what follows, a bounded domain $\Omega\subset\mathbb{R}^3$ is said to be \emph{strong local Lipschitz} if it has a \emph{locally Lipschitz boundary} (cf. \cite{AdF03}); that is, for every $\mathbf{x}\in\partial\Omega$, there exists a neighborhood $\mathcal{M}_{\epsilon}$ of $\mathbf{x}$ such that $\mathcal{M}_{\epsilon}\cap\partial\Omega$ is the graph of a Lipschitz continuous function. The next theorem, which shall be proven in {Section~\ref{sec:3}}, is a crucial ingredient for the proof of Theorem~\ref{thm:main}.
\begin{thm}\label{thm:thm2.1}
	Given any bounded domain $\Omega\subset\mathbb{R}^3$ satisfying the strong local Lipschitz condition, let $\gamma,\mu\in{C^3(\Omega)}\cap C^0(\overline\Omega)$ be nowhere vanishing in $\overline{\Omega}$.
	Let $\bm{\zeta},\bm{\eta}\in\mathbb{C}^3\backslash\{\mathbf{0}\}$ be such that $\bm{\zeta}\cdot\bm{\zeta}=0$ and $\bm{\zeta}\cdot\bm{\eta}=0$.
	For any $p>6$, and any constants $c^\mathbf{E}_{j}, c^{\mathbf{H}}_{j}$, $j=1,2$, that are independent of $|\bm{\zeta}|$, the Maxwell system \eqref{eq:MaxwellEH_full}
	admits a solution $(\mathbf{E},\mathbf{H})$ of the form
	\begin{equation}\label{eq:EHform}
	\mathbf{E}=e^{-\bm{\zeta}\cdot \mathbf{x}}\left(\gamma^{-1/2}\mathbf{E}_0+\tilde{\mathbf{E}}_{\bm{\zeta},\mathbf{E}_0}\right),\quad  \mathbf{H}=e^{-\bm{\zeta}\cdot \mathbf{x}}\left(\mu^{-1/2}\mathbf{H}_0+\tilde{\mathbf{H}}_{\bm{\zeta},\mathbf{H}_0}\right),
	\end{equation}
	where
	\begin{equation}\label{eq:EH0}
	\mathbf{V}_0=c^\mathbf{V}_{1}\hat{\bm{\zeta}}+c^\mathbf{V}_{2}(\bm{\eta}\wedge\hat{\bm{\zeta}}),
	\end{equation}
	with $\mathbf{V}=\mathbf{E},\mathbf{H}$.
	Moreover, there exists some constant $\delta>0$ such that
	\begin{equation}\label{eq:eq2.4}
	\|\tilde{\mathbf{E}}_{\bm{\zeta},\mathbf{E}_0}\|_{L_p(\Omega)^3}\le \frac{C}{|\bm{\zeta}|^{3/p+\delta}}\quad\text{and}\quad\|\tilde{\mathbf{H}}_{\bm{\zeta},\mathbf{H}_0}\|_{L_p(\Omega)^3}\le \frac{C}{|\bm{\zeta}|^{3/p+\delta}}.
	\end{equation}
\end{thm}	

The solutions of the particular forms in \eqref{eq:EHform} are referred to as the complex-geometric-optics (CGO) solutions to the Maxwell system. We proceed to show that

\begin{thm}\label{thm:thm2Ein}
	Let $\mathbf{V}=\mathbf{E}^0$ or $\mathbf{V}=\mathbf{H}^0$ where $(\mathbf{E}^0,\mathbf{H}^0)$ is the same as in Lemma \ref{lem:analy} with $\mathbf{x}_0=\mathbf{0}$.
	Then $\mathbf{V}$ can be written in a neighborhood $\mathcal{N}_{\epsilon}$ of $\mathbf{0}$ as
	\begin{equation}\label{eq:Vform}
		\mathbf{V}=\mathbf{P}_{N_{\mathbf{V}}}+\mathbf{M}_{N_{\mathbf{V}}+1} \mathbf{R}_{\mathbf{V}},
	\end{equation}
	such that
	\begin{enumerate}
		\item $\mathbf{P}_{N}$ is a non-trivial $3$-dimensional homogeneous harmonic polynomial of order $N$, which satisfies
		\begin{equation}\label{eq:eq2.11}
		\nabla\cdot \mathbf{P}_{N_{\mathbf{V}}}=0;
		\end{equation}\label{req:P1}
		\item $\mathbf{M}_{N_{\mathbf{V}}+1}$ is a diagonal matrix:
		$$\mathbf{M}_{N_{\mathbf{V}}+1}=\mathrm{diag}\left({M}_{N_{\mathbf{V}}+1}^{(j)}\right)_{j=1}^3$$
		with ${M}_{N_{\mathbf{V}}+1}^{(j)}$, $j=1,2,3$, homogeneous polynomials of order not less than $N_{\mathbf{V}}+1$; \label{req:M}
		\item $\mathbf{R}_{\mathbf{V}}$ is a $3$-dimensional vector fields {bounded in $\mathcal{N}_{\epsilon}$;}\label{req:RV}
		\item recalling the notation of the set $\mathcal{S}$ defined in \eqref{eq:notS}, one has
		\begin{equation}\label{eq:eq2.14}
		\nabla\wedge \mathbf{P}_{N_\mathbf{S}}=0,\quad \forall\mathbf{S}\in \mathcal{S}.
		\end{equation}
	\end{enumerate}	
%
\end{thm}

\begin{proof}
	Lemma~\ref{lem:analy} shows that each Cartesian component of $\mathbf{V}$ is analytic and hence can be represented as a Taylor series in $\mathcal{N}_{\epsilon}$.
	Let $\mathbf{P}_{N_{\mathbf{V}}}=(P_{N_{\mathbf{V}}}^{(j)})_{j=1}^3$ be the {lowest-order homogeneous polynomial} of the Taylor series of $\mathbf{V}$ at $\mathbf{x}_0=\mathbf{0}$ with the {lowest order} $N_{\mathbf{V}}$, as is introduced in Definition~\ref{defn:defn2.1}.	
	It is known for each $j=1,2,3$ that $V^{(j)}$ satisfies in $\mathcal{N}_{\epsilon}$ the Helmholtz equation (cf. \cite{CoK12})
	\begin{equation}
	\Delta V^{(j)}+k^2V^{(j)}=0
	\end{equation}
	with $k$ given by \eqref{eq:k}. Then the harmonicity of $\mathbf{P}_{N_{\mathbf{V}}}$ is readily seen by \cite[Lemma 2.4]{BPS14}.
	
	As for the remainder term $\mathbf{M}_{N_{\mathbf{V}}+1} \mathbf{R}_{\mathbf{V}}$ in \eqref{eq:Vform}, set $M_{N_{\mathbf{V}}+1}^{(j)}=R_{\mathbf{V}}^{(j)}=0$ if $V^{(j)}-P_N^{(j)}=0$, $j=1,2,3$; otherwise
	let $M_{N_{\mathbf{V}}+1}^{(j)}$ be the lowest-order homogeneous polynomial of $V^{(j)}-P_{N_{\mathbf{V}}}^{(j)}$ (whose order might be larger than $N_{\mathbf{V}}+1$), and let
	\begin{equation}
	R_{\mathbf{V}}^{(j)}:=\left( V^{(j)}-P_{N_{\mathbf{V}}}^{(j)}\right) / M_{N_{\mathbf{V}}+1}^{(j)},
	\end{equation}
	for $j=1,2,3$. Then $\mathbf{R}_{\mathbf{V}}=(R_{\mathbf{V}}^{(j)})_{j=1}^3$ is bounded in $\mathcal{N}_{\epsilon}$ and $\mathbf{V}$ can be written in the form \eqref{eq:Vform}.	
	
	Noting that the identities \eqref{eq:eq2.11} and \eqref{eq:eq2.14} trivially hold whenever one has $N_{\mathbf{V}}=0$ or $\mathbf{V}=\mathbf{P}_{N_{\mathbf{V}}}$, we rule out these two cases in the following arguments.
	Notice for small $|\mathbf{x}|$ that
	\begin{equation}\label{eq:eq3.150}
	 \left( \mathbf{V}-\mathbf{P}_{N_{\mathbf{V}}}\right) (\mathbf{x})= \mathcal{O}(|\mathbf{x}|^{s+1}),\quad s\ge N_\mathbf{V}.
	\end{equation}
	If \eqref{eq:eq2.11} is not true, then the divergence-free property of $\mathbf{V}$ in combination with the relation \eqref{eq:eq3.150} implies that
	\begin{equation}\label{eq:eq2.190}
	\nabla\cdot \mathbf{P}_{N_{\mathbf{V}}}(\mathbf{x})
	=\mathcal{O}(|\mathbf{x}|^{s}),\quad s\ge N_\mathbf{V},
	\end{equation}
	which cannot hold since $\mathbf{P}_{N_{\mathbf{V}}}(\mathbf{x})
	=\mathcal{O}(|\mathbf{x}|^{N_{\mathbf{V}}})$ for sufficiently small $|\mathbf{x}|$.

	Finally for the identity \eqref{eq:eq2.14}, we
	assume without loss of generality that $N_{\mathbf{E}}:=N_{\mathbf{E}^0}\le N_{\mathbf{H}}:= N_{\mathbf{H}^0}$.
	If $\nabla\wedge \mathbf{P}_{N_\mathbf{E}}$ does not vanish identically, then
	\begin{equation}
	\mathtt{i}\omega\mu_0 \mathbf{H}^0(\mathbf{x})=\nabla\wedge \mathbf{P}_{N_\mathbf{E}}(\mathbf{x})+\nabla\wedge (\mathbf{M}_{N_{\mathbf{E}}+1} \mathbf{R}_{\mathbf{E}})(\mathbf{x})= \mathcal{O}(|\mathbf{x}|^{N_\mathbf{E}-1}),
	\end{equation}
	and thus
	\begin{equation}
	N_\mathbf{H}= N_\mathbf{E}-1,
	\end{equation}
	which contradicts the hypothesis $N_\mathbf{H}\ge N_\mathbf{E}$.
	
	The proof is complete.
\end{proof}

\begin{cor}
	The vector field $\mathbf{V}$ given in Theorem~\ref{thm:thm2Ein} can be also represented as
	\begin{equation}\label{eq:Vform1}
	\mathbf{V}= \mathbf{M}_{N_\mathbf{V}}\tilde{\mathbf{V}},
	\end{equation}
	where $ \mathbf{M}_{N_\mathbf{V}}$ is a diagonal $3\times 3$ matrix satisfying the condition (\ref{req:M}) in Theorem~\ref{thm:thm2Ein} with the integer $N_\mathbf{V}+1$ replaced by $N_\mathbf{V}$, and $\tilde{\mathbf{V}}$ satisfies (\ref{req:RV}).
\end{cor}

In what follows, we let $\mathscr{L}$ denote the Laplace transform operator defined by
\begin{equation}
\mathscr{L}[f](\bm{\zeta}):=\int_{\mathcal{K}}e^{-\mathbf{x}\cdot\bm{\zeta}}f(\mathbf{x})d\mathbf{x},
\end{equation}
where $\mathcal{K}$ is the positive orthant given in \eqref{eq:orth1}.
The following theorem is of crucial importance for our proof of Theorem~\ref{thm:main}, which shall be verified in Section~\ref{sec:LapTrans}.
\begin{thm}\label{thm:LapTrans}
	Let $(\mathbf{E}^0,\mathbf{H}^0)$ be a non-trivial pair which satisfies \eqref{eq:EHin_eqn} in a neighborhood of $\mathbf{x}_0=\mathbf{0}$. Let $\mathbf{P}=\mathbf{P}_{N_\mathbf{S}}$ be the nonzero homogeneous harmonic polynomial given by Theorem~\ref{thm:thm2Ein}.
	Then $\bm{\zeta}\cdot\mathscr{L}[\mathbf{P}](\bm{\zeta})$ cannot vanish identically on any open subset of the variety $\bm{\zeta}\cdot\bm{\zeta}=0$, unless $(\mathbf{E}^0,\mathbf{H}^0)$ is of the inadmissible type given in Definition~\ref{defn:admiType} or equivalently, in Theorem~\ref{thm:admiEquiv}.
\end{thm}

\subsection{Proof of Theorem~\ref{thm:main}}

Throughout the rest of the paper, we denote by $q'$ the H\"{o}lder conjugate exponent of any given constant $q\in\mathbb{R}$, namely, $q'\in\mathbb{R}$ such that $1/q+1/q'=1$.

\begin{proof}[Proof of Theorem~\ref{thm:main}]

We prove Theorem~\ref{thm:main} by contradiction.
Suppose that $(\mathbf{E}^0,\mathbf{H}^0)$ can be extended as a solution to \eqref{eq:EHin_eqn} into a neighborhood $\mathcal{N}_{\epsilon}$ of the corner $\mathbf{0}$.
Then by Theorem~\ref{thm:thm2Ein}, $\mathbf{E}^0$ and $\mathbf{H}^0$ can be presented in $\mathcal{N}_{\epsilon}$ as
\begin{equation}\label{eq:eq2.12}
\mathbf{E}^0=\mathbf{P}_{N_\mathbf{E}}+\mathbf{M}_{N_\mathbf{E}+1}\mathbf{R}_\mathbf{E},\quad \mathbf{H}^0=\mathbf{P}_{N_\mathbf{H}}+\mathbf{M}_{N_\mathbf{H}+1} \mathbf{R}_\mathbf{H} ,
\end{equation}
where $\mathbf{P}_{N_\mathbf{E}}$ and $\mathbf{P}_{N_\mathbf{H}}$ are three dimensional vector fields satisfying the condition (\ref{req:P1}) in Theorem~\ref{thm:thm2Ein}, $\mathbf{M}_{N_\mathbf{E}+1}$ and $\mathbf{M}_{N_\mathbf{H}+1}$ are $3\times 3$ diagonal matrices satisfying (\ref{req:M}), and $\mathbf{R}_\mathbf{E}$ and $\mathbf{R}_\mathbf{H}$ satisfy (\ref{req:RV}).

Assume without loss of generality that $N:=N_\mathbf{E}\le N_\mathbf{H}$.
In the rest of the proof, we shall denote for notational simplicity that
$\mathbf{P}_N:=\mathbf{P}_{N_\mathbf{E}}$ and $\mathbf{M}_{N+1}:=\mathbf{M}_{N_\mathbf{E}+1}$.
Let $\Omega\subset\mathbb{R}^3$ be a bounded domain containing $D$ and satisfying the strong local Lipschitz condition. Given any $p>6$ and any $\bm{\zeta},\bm{\eta}\in\mathbb{C}^3\backslash\{\mathbf{0}\}$ such that $\bm{\zeta}\cdot\bm{\zeta}=0$ and $\bm{\zeta}\cdot\bm{\eta}=0$, let $(\mathbf{E},\mathbf{H})$ be a pair of CGO solutions defined in Theorem~\ref{thm:thm2.1} with the domain $\Omega \supset D$, and the constants $c_2^{\mathbf{E}}=c^{\mathbf{H}}_j=0$, $j=1,2$.
Then $\mathbf{E}$ and $\mathbf{H}$ are of the form
\begin{equation}\label{eq:eq2.13}
\mathbf{E}=e^{-\bm{\zeta}\cdot \mathbf{x}}\left(\gamma^{-1/2}\mathbf{E}_0+\tilde{\mathbf{E}}\right),\quad  \mathbf{H}=e^{-\bm{\zeta}\cdot \mathbf{x}}\tilde{\mathbf{H}},
\end{equation}
where
\begin{equation}\label{eq:E0}
\mathbf{E}_0=c_{1}\hat{\bm{\zeta}},
\end{equation}
and
$(\tilde{\mathbf{E}},\tilde{\mathbf{H}})=(\tilde{\mathbf{E}}_{\bm{\zeta},\mathbf{E}_0},\tilde{\mathbf{H}}_{\bm{\zeta},\mathbf{H}_0})$ satisfies the estimates \eqref{eq:eq2.4} with some constant $\delta>0$.

Let $\gamma(\mathbf{0})=c_e\varepsilon_0$ with $c_e\neq 0,1$.
Denote
\begin{equation}
\tilde{\gamma}:=\gamma-\varepsilon_0,\quad \tilde{\mu}:=\mu-\mu_0.
\end{equation}
Let $c_0$ be the value of $\tilde{\gamma}\gamma^{-1/2}$ at $\mathbf{x}_0=\mathbf{0}$, namely,
\begin{equation}
c_0=(c_e-1)c_e^{-1/2}\varepsilon_0^{1/2}\neq 0.
\end{equation}
Inserting \eqref{eq:eq2.12} and \eqref{eq:eq2.13} into the orthogonality \eqref{eq:ortho} yields,
\begin{equation}\label{eq:eqI}
0=I_0+I_1+I_2+{I}_3,
\end{equation}
where $I_j$, $0\le j \le 3$, are integrals given respectively by
\begin{equation}
I_0:= c_0\int_{{\mathcal{N}}^+_{\epsilon}} e^{-\mathbf{x}\cdot \bm{\zeta}}\, \mathbf{E}_0\cdot \mathbf{P}_{N},
\end{equation}
\begin{equation}
I_1:=\int_{D\setminus {\mathcal{N}}^+_{\epsilon}} \tilde{\gamma}e^{-\mathbf{x}\cdot\bm{\zeta}}\, \mathbf{E}^0\cdot(\gamma^{-1/2}\mathbf{E}_0+\tilde{\mathbf{E}})
-\int_{D\setminus {\mathcal{N}}^+_{\epsilon}} \tilde{\mu}e^{-\mathbf{x}\cdot\bm{\zeta}}\, \mathbf{H}^0\cdot \tilde{\mathbf{H}},
\end{equation}
\begin{equation}\label{eq:I2}
\begin{split}
I_2:=&\int_{{\mathcal{N}}^+_{\epsilon}} \gamma^{-1/2}e^{-\mathbf{x}\cdot \bm{\zeta}}\, \mathbf{E}_0\cdot
\left(  (\tilde{\gamma}-c_0\gamma^{1/2}) \mathbf{P}_{N}+\tilde{\gamma} \mathbf{M}_{N+1}\mathbf{R}_\mathbf{E}\right) ,
\end{split}
\end{equation}
and
\begin{equation}
{I}_3:=\int_{{\mathcal{N}}^+_{\epsilon}} \tilde{\gamma}e^{-\mathbf{x}\cdot \bm{\zeta}} \tilde{\mathbf{E}}\cdot \mathbf{E}^0
-\int_{{\mathcal{N}}^+_{\epsilon}} \tilde{\mu}e^{-\mathbf{x}\cdot \bm{\zeta}} \tilde{\mathbf{H}}\cdot \mathbf{H}^0.
\end{equation}
In what follows, we shall estimate the items in the RHS of \eqref{eq:eqI}, respectively, and derive a contradiction.

\medskip

\noindent \textbf{Part I: Estimate of $I_0$. }

\medskip
	
Denote the variety
\begin{equation}
\mathcal{U}:=\{\bm{\zeta}\in\mathbb{C}^3;\, \bm{\zeta}\cdot\bm{\zeta}=0\}.
\end{equation}
Given a constant $c\in(0,1/\sqrt{6})$, define
\begin{equation}
\tilde{\mathcal{U}}_c=\{\bm{\zeta}=(\zeta^{(j)})_{j=1}^3 \in\mathbb{C}^3;\,\min_{j} \Re \frac{\zeta^{(j)}}{|\bm{\zeta}|}>c\},
\end{equation}
and
\begin{equation}
\mathcal{U}_c=\mathcal{U}\cap\tilde{\mathcal{U}}_c\cap \{\bm{\zeta}\in\mathbb{C}^3;|\bm{\zeta}|>1\}.
\end{equation}
The arguments in \cite{BPS14} indicate that $\mathcal{U}_c$ is a non-empty open subset of $\mathcal{U}$.
Therefore, Theorem~\ref{thm:LapTrans} implies that there exists $\bm{\zeta}^*\in \mathcal{U}_c$ such that
\begin{equation}
\bm{\zeta}^*\cdot\mathscr{L}[\mathbf{E}_0\cdot \mathbf{P}_{N}](\bm{\zeta}^*)\neq 0.
\end{equation}
We fix $\bm{\zeta}^*$ and set in the sequel that $\bm{\zeta}$ is always of the form
\begin{equation}\label{eq:zeta}
\bm{\zeta}=|\bm{\zeta}|\frac{\bm{\zeta}^*}{|\bm{\zeta}^*|}.
\end{equation}
Let $\bm{\zeta}=r\bm{\zeta}_r$ with $r$ a positive constant. It is noticed that
	\begin{eqnarray}
	\mathscr{L}[\mathbf{E}_0\cdot \mathbf{P}_{N}](\bm{\zeta})&=&\int_{\mathcal{K}} e^{-\mathbf{x}\cdot \bm{\zeta}}\, \mathbf{E}_0\cdot \mathbf{P}_{N}(\mathbf{x})\,d\mathbf{x}\\
	&
	{=} &	
	r^{-(N+3)}\int_{\mathcal{K}} e^{-\mathbf{y}\cdot\bm{\zeta}_r}\, \mathbf{E}_0\cdot \mathbf{P}_{N}(\mathbf{y})\,d\mathbf{y}\\
	& = & r^{-(N+3)}\mathscr{L}[\mathbf{E}_0\cdot \mathbf{P}_{N}](\bm{\zeta}_r).
	\end{eqnarray}
	As a consequence, we have for any $\bm{\zeta}$ of the form \eqref{eq:zeta}
	that
	\begin{equation}\label{eq:eqLP}
	\left| \mathscr{L}[\mathbf{E}_0\cdot \mathbf{P}_{N}](\bm{\zeta}) \right|
	=\left( \frac{|\bm{\zeta}^*|}{|\bm{\zeta}|}\right)^{N+3} \left|\mathscr{L}[\mathbf{E}_0\cdot \mathbf{P}_{N}](\bm{\zeta}^*) \right| \\
	=: \frac{C}{|\bm{\zeta}|^{N+3}},
	\end{equation}
	with the constant $C$ depending on $N$ and $|\bm{\zeta}^*|$.

For sufficiently large $|\bm{\zeta}|$, it is verified by \cite[Equation (18)]{BPS14} that
\begin{equation}
\left| \int_{\mathbf{x}\in\mathcal{K},|\mathbf{x}|>{\epsilon}} e^{-\mathbf{x}\cdot \bm{\zeta}}\, \mathbf{E}_0\cdot \mathbf{P}_{N}\right|
\leq\mathcal{O}\left(e^{-c \epsilon |\bm{\zeta}|}\right),
\end{equation}
which in combination with \eqref{eq:eqLP} yields
\begin{equation}\label{eq:eqI0}
|I_0|=\mathcal{O}\left(|\bm{\zeta}|^{-(N+3)}\right).
\end{equation}
	
\medskip

\noindent \textbf{Part II: Estimate of $I_1$. }

\medskip

For $\mathbf{x}\in\mathbb{R}^3$ in the positive orthant, and $\bm{\zeta}\in\mathbb{C}^3$ given in \eqref{eq:zeta},
		\begin{equation}
		\Re (\mathbf{x}\cdot \bm{\zeta})= \mathbf{x}\cdot  \Re\bm{\zeta} \ge |\mathbf{x}|\min_{j} \Re \zeta_j> c |\bm{\zeta}| |\mathbf{x}|.
		\end{equation}
		Then, we have
		\begin{equation}\label{eq:I1}
		\begin{split}
		\left| I_1 \right|
		=&\left|\int_{D\setminus \mathcal{N}_{\epsilon}} \tilde{\gamma}e^{-\mathbf{x}\cdot\bm{\zeta}}\, \mathbf{E}^0\cdot(\gamma^{-1/2}\mathbf{E}_0+\tilde{\mathbf{E}})
		-\int_{D\setminus \mathcal{N}_{\epsilon}} \tilde{\mu}e^{-\mathbf{x}\cdot\bm{\zeta}}\, \mathbf{H}^0\cdot \tilde{\mathbf{H}}\right|\\
		\leq &\int_{D\setminus \mathcal{N}_{\epsilon}} e^{-\Re (\mathbf{x}\cdot \bm{\zeta})} \left(  \left| \tilde{\gamma}\, \mathbf{E}^0\cdot (\gamma^{-1/2}\mathbf{E}_0+\tilde{\mathbf{E}})\right|+\left| \tilde{\mu}\mathbf{H}^0\cdot \tilde{\mathbf{H}}\right|\right)\\
		\le  &e^{-c \epsilon |\bm{\zeta}|}\left(\|\tilde{\gamma}\gamma^{-1/2}\mathbf{E}^0\cdot \mathbf{E}_0\|_{L^1(D)}+\|\tilde{\mathbf{E}}\|_{L^p(D)^3}\|\tilde{\gamma}\mathbf{E}^0\|_{L^{p'}(D)^3}\right) \\
		&+e^{-c \epsilon |\bm{\zeta}|}\|\tilde{\mathbf{H}}\|_{L^p(D)^3}\|\tilde{\mu}\mathbf{H}^0\|_{L^{p'}(D)^3} \\
		\le & C e^{-c \epsilon |\bm{\zeta}|}.
		\end{split}
		\end{equation}
		
\medskip

\noindent \textbf{Part III: Estimate of $I_2$. }

\medskip		
		
Notice that $$\tilde{\gamma}(\mathbf{0})-c_0\gamma^{1/2}(\mathbf{0})=0.$$
		Then by the Taylor series expansion of $(\tilde{\gamma}-c_0\gamma^{1/2})$ around the corner $\mathbf{x}_0=\mathbf{0}$, 
		the integral $I_2$ in \eqref{eq:I2} can actually be regarded as
		\begin{equation}
		I_2
		=\int_{{\mathcal{N}}^+_{\epsilon}} \gamma^{-1/2} e^{-\mathbf{x}\cdot \bm{\zeta}} \, \mathbf{E}_0\cdot\left( \tilde{\mathbf{Q}}^{N+1} \mathbf{F}\right) ,
		\end{equation}
		where $\tilde{\mathbf{Q}}^{N+1}$ is a $3\times 3$ diagonal matrix whose diagonal entries are homogeneous polynomials of degree not less than $N+1$, and $\mathbf{F}$ is bounded in $\mathcal{N}_\epsilon$ with $\epsilon\le \epsilon_0$.
		
		It is observed by the fundamental calculus that
		\begin{equation}\label{eq:eq2.15}
		\begin{split}
		I_2&=\int_{{\mathcal{N}}^+_{\epsilon}} \gamma^{-1/2}e^{-|\bm{\zeta}|\mathbf{x}\cdot  \hat{\bm{\zeta}}} \, \mathbf{E}_0\cdot\left( \tilde{\mathbf{Q}}^{N+1} \mathbf{F}\right) \,d\mathbf{x}\\
		&=\frac{1}{|\bm{\zeta}|^{N+4}}\int_{\mathcal{N}_{|\bm{\zeta}|\epsilon}} e^{-\mathbf{y}\cdot\hat{\bm{\zeta}}} \,\gamma^{-1/2}(\mathbf{y}/|\bm{\zeta}|)\,
		 \tilde{\mathbf{Q}}^{N+1}(\mathbf{y}) \mathbf{F}(\mathbf{y}/|\bm{\zeta}|) \,d\mathbf{y}.
		\end{split}
		\end{equation}
		Thus
		\begin{equation}\label{eq:eq2.16}
		\begin{split}
		|I_2|& \le \frac{1}{|\bm{\zeta}|^{N+4}}\|\gamma^{-1/2}\mathbf{F}\|_{L^{\infty}(\mathcal{N}^+_{\epsilon_0})^3} \int_{\mathcal{K}} e^{-\mathbf{y}\cdot\Re\hat{\bm{\zeta}}} \, \left| \tilde{\mathbf{Q}}^{N+1}(\mathbf{y})\right|  \,d\mathbf{y}\\
		& \le \frac{C
			}{|\bm{\zeta}|^{N+4}},
		\end{split}
		\end{equation}
		where the integral term in \eqref{eq:eq2.16} is bounded since we have
		\begin{equation}
		\mathbf{y}\cdot\Re \hat{\bm{\zeta}}\ge \tau |\mathbf{y}|.
		\end{equation}
		
\medskip

\noindent \textbf{Part IV: Estimate of $I_3$. }

\medskip

We write
		\begin{equation}
		\mathbf{E}^0=\mathbf{M}_{N_\mathbf{E}} \tilde{\mathbf{E}}_P,\quad
		\mathbf{H}^0=\mathbf{M}_{N_\mathbf{H}} \tilde{\mathbf{H}}_P,
		\end{equation}
		as in \eqref{eq:Vform1}, where $\tilde{\mathbf{E}}_P$ and $\tilde{\mathbf{H}}_P$ are bounded in $\mathcal{N}_\epsilon$ for $\epsilon\le\epsilon_0$, and $\mathbf{M}_{N_\mathbf{E}}$ and $\mathbf{M}_{N_\mathbf{H}}$ are $3\times 3$ diagonal matrices satisfying the condition (\ref{req:M}) in Theorem~\ref{thm:thm2Ein} with integers $N_\mathbf{E}$ and $N_\mathbf{H}$, respectively.
		Then 
		\begin{equation}
			\begin{split}
			|\mathbf{I}_3|
			=& \left| \int_{{\mathcal{N}}^+_{\epsilon}} \tilde{\gamma}e^{-\mathbf{x}\cdot \bm{\zeta}} \tilde{\mathbf{E}}_P\cdot \mathbf{M}_{N_\mathbf{E}} \tilde{\mathbf{E}}
			-\int_{{\mathcal{N}}^+_{\epsilon}} \tilde{\mu}e^{-\mathbf{x}\cdot \bm{\zeta}} \tilde{\mathbf{H}}_P\cdot \mathbf{M}_{N_\mathbf{H}} \tilde{\mathbf{H}} \right| \\
			\leq &\|\tilde{\gamma}\tilde{\mathbf{E}}_P\|_{L^{\infty}(\mathcal{N}^+_{\epsilon_0})^3} \int_{{\mathcal{N}}^+_{\epsilon}}\left|  e^{-\mathbf{x}\cdot \bm{\zeta}}\, \mathbf{M}_{N_\mathbf{E}} \tilde{\mathbf{E}} \right| \\
			&+\|\tilde{\mu}\tilde{\mathbf{H}}_P\|_{L^{\infty}(\mathcal{N}^+_{\epsilon_0})^3} \int_{{\mathcal{N}}^+_{\epsilon}}\left|  e^{-\mathbf{x}\cdot \bm{\zeta}}\, \mathbf{M}_{N_\mathbf{H}} \tilde{\mathbf{H}} \right| \\
			:=&\|\tilde{\gamma}\tilde{\mathbf{E}}_P\|_{L^{\infty}(\mathcal{N}^+_{\epsilon_0})^3}I^{\mathbf{E}}_3+\|\tilde{\mu}\tilde{\mathbf{H}}_P\|_{L^{\infty}(\mathcal{N}^+_{\epsilon_0})^3} I^{\mathbf{H}}_3.
			\end{split}
		\end{equation}
		Applying the similar strategy as in \eqref{eq:eq2.15} and \eqref{eq:eq2.16}, we have
		\begin{equation}\label{eq:eq2.19}
		\begin{split}
		I^{\mathbf{E}}_3&=\frac{1}{|\bm{\zeta}|^{N+3}}\int_{\mathcal{N}_{|\bm{\zeta}|\epsilon}} \left| e^{-\mathbf{y}\cdot\hat{\bm{\zeta}}} \, \mathbf{M}_{N_\mathbf{E}}(\mathbf{y}) \tilde{\mathbf{E}}(\mathbf{y}/|\bm{\zeta}|)\right|  \,d\mathbf{y}\\
		& \le \frac{\|\mathbf{F}_{\mathbf{E}}\|_{L^{p'}}}{|\bm{\zeta}|^{N+3}}  \|\tilde{\mathbf{E}}(\cdot/|\bm{\zeta}|)\|_{L^p(\mathcal{N}_{|\bm{\zeta}|\epsilon})^3}\\
		&=\frac{\|\mathbf{F}_{\mathbf{E}}\|_{L^{p'}}}{|\bm{\zeta}|^{N+3-3/p}}  \|\tilde{\mathbf{E}}\|_{L^p({\mathcal{N}}^+_{\epsilon})^3},
		\end{split}
		\end{equation}
		where the function $\mathbf{F}_{\mathbf{E}}$ is given by $\mathbf{F}_\mathbf{E}(\mathbf{x}):=e^{-\mathbf{x}\cdot  \hat{\bm{\zeta}}} \,\mathbf{M}_{N_\mathbf{E}}(\mathbf{x})$, and the associated $L^{p'}$ norm is taken on the whole positive orthant.
		Similarly,
		\begin{equation}
		\begin{split}
		I^{\mathbf{H}}_3&\le\frac{C}{|\bm{\zeta}|^{N_\mathbf{H}+3-3/p}}  \|\tilde{\mathbf{H}}\|_{L^p({\mathcal{N}}^+_{\epsilon})^3},
		\end{split}
		\end{equation}
		Therefore, applying the estimates \eqref{eq:eq2.4} for $\tilde{\mathbf{E}}$ and $\tilde{\mathbf{H}}$, and noting that $N_\mathbf{H}\le N_\mathbf{E}=N$, we obtain
		\begin{equation}\label{eq:I3}
		|\mathbf{I}_3|\le \frac{C}{|\bm{\zeta}|^{N+3+\delta}} .
		\end{equation}

Summing up the above estimates, we arrive at
	\begin{equation}
	I_1+I_2+{I}_3\le o\left(  \frac{1}{|\bm{\zeta}|^{N+3}} \right),
	\end{equation}
	which, however, is a contradiction to \eqref{eq:eqI0} and \eqref{eq:eqI}.
	
	The proof is complete.
\end{proof}

\begin{proof}[Proof of Theorem~\ref{thm:thm1.2}]
	Assume, for example, $\mu\equiv \mu_0$. Then the orthogonality \eqref{eq:ortho} reduces to
	\begin{equation}
	\int_{D}(\gamma-\varepsilon_0)\mathbf{E}^0\cdot \mathbf{E}=0.
	\end{equation}
	Following the same arguments, but only with $\mathbf{E}$, $\gamma$ and other related fields, as in the proof of Theorem~\ref{thm:main}, one can prove the statement in Theorem~\ref{thm:thm1.2}.
\end{proof}

\section{Construction of CGO solutions}\label{sec:3}

The current section aims at proving Theorem~\ref{thm:thm2.1}.

\subsection{A matrix representation of the Maxwell equations}\label{sec:martixForm} 

Throughout this section, we let the functions $\gamma$, $\mu$ and the domain $\Omega$ be the same as in Theorem~\ref{thm:thm2.1}.
We shall apply the $8\times 8$ matrix form of the Maxwell equations
\begin{align}\label{eq:MaxwellEH}
\nabla\wedge \mathbf{E}-\mathtt{i}\omega \mu \mathbf{H}=0,\qquad&
\nabla\wedge \mathbf{H}+\mathtt{i}\omega\gamma \mathbf{E}=0.
\end{align}
as that in \cite{COS09}, which was originally introduced by Ola and Somersalo in \cite{OlS96} and shall be further discussed in what follows.

Define the matrix (differential) operator $\mathscr{P}^{\mp}$ by
\begin{equation}
\mathscr{P}^{\mp}(\bm{\xi})=
\left(
\begin{array}{cc}
0 & \mathscr{P}^{-}(\bm{\xi})\\
\mathscr{P}^{+}(\bm{\xi}) &0
\end{array}
\right),
\end{equation}
with
\begin{equation}
\mathscr{P}^{+}(\bm{\xi})=
\left(
\begin{array}{cc}
0& \bm{\xi} \cdot\\
\bm{\xi} &\bm{\xi} \wedge
\end{array}
\right)\quad \text{and}\quad
\mathscr{P}^{-}(\bm{\xi})=
\left(
\begin{array}{cc}
0& \bm{\xi} \cdot\\
\bm{\xi} &-\bm{\xi} \wedge
\end{array}
\right),
\end{equation}
for a 3-dimensional field $\bm{\xi}$.
We also denote
\begin{equation}
\mathscr{P}^{\mp}(\bm{\xi},\bm{\eta})=
\left(
\begin{array}{cc}
0 & \mathscr{P}^{-}(\bm{\xi})\\
\mathscr{P}^{+}(\bm{\eta}) &0
\end{array}
\right)\quad\text{and}\quad
\mathscr{P}^{\pm}(\bm{\xi},\bm{\eta})=
\left(
\begin{array}{cc}
0 & \mathscr{P}^{+}(\bm{\xi})\\
\mathscr{P}^{-}(\bm{\eta}) &0
\end{array}
\right).
\end{equation}
\begin{rem}
	It is noticed that
	\begin{equation}\label{eq:eq3.5}
	\mathscr{P}^{\mp}(\nabla)^2=\Delta \mathbf{I}_8,
	\end{equation}
	and that
	\begin{equation}\label{eq:eq2.6}
	\begin{split}
	\mathscr{P}^{\pm}(\bm{\xi},\bm{\eta})\mathscr{P}^{\mp}(\bm{\eta},\xi)
	&=\mathrm{diag}\left(\bm{\xi} \cdot\bm{\xi},\bm{\eta}\cdot\bm{\eta}\right)\\
	&:=\mathrm{diag}\big((\bm{\xi} \cdot \bm{\xi} ) \mathbf{I}_4,(\bm{\eta}\cdot\bm{\eta}) \mathbf{I}_4\big).
	\end{split}
	\end{equation}
\end{rem}

For an 8-dimensional field $\mathcal{X}$, we next consider the matrix differential operator equation
\begin{equation}\label{eq:Xeqn}
\left(\mathscr{P}^{\mp}(\nabla)+\mathcal{V}_{\mu,\gamma}\right) \mathcal{X}=0,
\end{equation}
where the $8\times 8$ matrix $\mathcal{V}_{\mu,\gamma}$ is defined by
\begin{equation}\label{eq:Vdefn}
\mathcal{V}_{\mu,\gamma}=\left(
\begin{array}{cccc}
\mathtt{i}\omega\mu& 0& 0 & \nabla \alpha\cdot\\
0& \mathtt{i}\omega\mu \mathbf{I}_3& \nabla\alpha& 0\\
0& \nabla\beta \cdot & \mathtt{i}\omega\gamma & 0\\
\nabla\beta& 0 & 0& \mathtt{i}\omega\gamma \mathbf{I}_3
\end{array}\right),
\end{equation}
with
\begin{equation}\label{eq:ab1}
 \alpha:=\log{\gamma}\ \ \mbox{and}\ \ \beta:=\log{\mu}.
 \end{equation}
Throughout the rest of the paper, a 8-dimensional vector $\mathcal{X}$ is also denoted by
\begin{equation}\label{eq:calX}
\mathcal{X}=\left(\mathcal{X}^{h},(\mathcal{X}^{\mathbf{H}})^T,\mathcal{X}^{e},(\mathcal{X}^{\mathbf{E}})^T\right)^T,
\end{equation}
with the scalar fields $\mathcal{X}^{h}$ and $\mathcal{X}^{e}$, and the $3$-dimensional vector fields $\mathcal{X}^{\mathbf{H}},\mathcal{X}^{\mathbf{E}}$.

It was proved that (cf. \cite{OlS96} and \cite{COS09}),

\begin{lem}\label{lem:lem2.1}
	Let $\Omega$ be any bounded domain in $\mathbb{R}^3$. For any 8-dimensional field $\mathcal{X}$ solving \eqref{eq:Xeqn} in $\Omega$, if $\mathcal{X}^{h}=\mathcal{X}^{e}=0$, then $\mathbf{E}:=\mathcal{X}^{\mathbf{E}}$ and $\mathbf{H}:=\mathcal{X}^{\mathbf{H}}$ satisfy in $\Omega$ the Maxwell equations \eqref{eq:MaxwellEH}.
\end{lem}

We shall also need the field
\begin{equation}\label{eq:Ydefn}
\mathcal{Y}=\mathrm{diag}\left(\mu^{1/2},\gamma^{1/2}\right) \mathcal{X},
\end{equation}
and the following result from \cite{COS09}. 

{
\begin{lem}\label{lem:k1}
	$\mathcal{X}$ satisfies \eqref{eq:Xeqn} if and only if $\mathcal{Y}$ satisfies
	\begin{equation}\label{eq:Yeqn}
	\left(\mathscr{P}^{\mp}(\nabla)+\mathcal{W}_{\mu,\gamma}\right) \mathcal{Y}=0,
	\end{equation}
	where
	\begin{equation}
	\mathcal{W}_{\mu,\gamma}=\mathtt{i}\kappa \mathbf{I}_8 +\frac{1}{2}\mathscr{P}^{\pm}(\nabla \alpha,\nabla\beta),
	\end{equation}
	with $\kappa:=\omega(\gamma\mu)^{1/2}$.
\end{lem}
\begin{rem}
	For the subsequent use, we remark that one can directly verify that
	\begin{equation}\label{eq:eq4.13}
	\nabla\kappa=\frac{1}{2}\kappa(\nabla\alpha+\nabla\beta),
	\end{equation}
	where $\kappa$ is introduced in Lemma~\ref{lem:k1} and $\alpha, \beta$ are defined in \eqref{eq:ab1}. 
\end{rem}
}

\begin{prop}\label{prop:Zeqn}
	Define
	\begin{equation}
	\mathcal{W}'_{\mu,\gamma}=\mathtt{i}\kappa \mathbf{I}_8 +\frac{1}{2}\mathscr{P}^{\pm}(\nabla\beta,\nabla \alpha).
	\end{equation}
	Let $\mathcal{Z}$ be a 8-dimensional field such that
	\begin{equation}
	\left(\mathscr{P}^{\mp}(\nabla)-\mathcal{W}'_{\mu,\gamma}\right)\mathcal{Z}=\mathcal{Y}.
	\end{equation}
	Then $\mathcal{Y}$ solves \eqref{eq:Yeqn} if and only if $\mathcal{Z}$ satisfies
	\begin{equation}\label{eq:Zeqn}
	\left(-\Delta \mathbf{I}_8+\mathcal{Q}_{\mu,\gamma} \right) \mathcal{Z}=0,
	\end{equation}
	where
	\begin{equation}
	\begin{split}
	\mathcal{Q}_{\mu,\gamma}=&-\kappa^2 \mathbf{I}_8+\frac{1}{4}\mathrm{diag}\big(\nabla \alpha\cdot \nabla \alpha, \, \nabla \beta\cdot \nabla \beta\big)\\
	&+\frac{1}{2}\mathscr{P}^{\mp}(\nabla_{\alpha},\nabla_{\beta})\mathscr{P}^{\pm}(\nabla\beta,\nabla \alpha)+ \mathtt{i}\left( \mathscr{P}^{\pm}(\nabla{\kappa})+\mathscr{P}^{\mp}(\nabla{\kappa})\right)
	\end{split}
	\end{equation}
	is compactly supported.
\end{prop}
\begin{rem}
	The $8\times 8$ matrix $\mathcal{Q}_{\mu,\gamma}$ has the form
	\begin{equation}
	\mathcal{Q}_{\mu,\gamma}=\mathrm{diag} (\mathcal{Q}_{\mu,\gamma}) +2\mathtt{i}\left(
	\begin{array}{cccc}
	0&0&0& \nabla\kappa\cdot\\ 0&0&\nabla\kappa& 0 \\
	0& \nabla\kappa\cdot&0&0\\ \nabla\kappa&0&0&0
	\end{array}\right),
	\end{equation}
	since
	\begin{equation}
	\begin{split}
	&\mathscr{P}^{\mp}(\nabla_{\alpha},\nabla_{\beta})\mathscr{P}^{\pm}(\nabla\beta,\nabla \alpha)\\
	=&\mathrm{diag}(\Delta\alpha,\, 2\nabla_{\alpha}(\nabla\alpha\cdot)-\Delta\alpha \mathbf{I}_3,\,
	\Delta\beta,\, 2\nabla_{\beta}(\nabla\beta\cdot)-\Delta\beta \mathbf{I}_3).
	\end{split}
	\end{equation}
\end{rem}

Proposition~\ref{prop:Zeqn} can be easily verified by using the following lemma.

\begin{lem}\label{lem:lem2.3}
	One has
	\begin{equation}\label{eq:eq2.20}
	\left(\mathscr{P}^{\mp}(\nabla)+\mathcal{W}_{\mu,\gamma}\right)\left(\mathscr{P}^{\mp}(\nabla)-\mathcal{W}'_{\mu,\gamma}\right)=\Delta \mathbf{I}_8-\mathcal{Q}_{\mu,\gamma},
	\end{equation}
	and
	\begin{equation}\label{eq:eq2.22}
	\left(\mathscr{P}^{\mp}(\nabla)-\mathcal{W}'_{\mu,\gamma}\right)\left(\mathscr{P}^{\mp}(\nabla)+\mathcal{W}_{\mu,\gamma}\right)=\Delta \mathbf{I}_8-\mathcal{Q}'_{\mu,\gamma},
	\end{equation}
	where $\mathcal{Q}_{\mu,\gamma}$ is the same as in Proposition~\ref{prop:Zeqn}, and $\mathcal{Q}'_{\mu,\gamma}$ is defined by
	\begin{equation}
	\begin{split}
	\mathcal{Q}'_{\mu,\gamma}=&-\kappa^2 \mathbf{I}_8+\frac{1}{4}\mathrm{diag}\big(\nabla \beta\cdot \nabla \beta,\, \nabla \alpha\cdot \nabla \alpha\big)\\
	&-\frac{1}{2}\mathscr{P}^{\mp}(\nabla_{\beta},\nabla_{\alpha})\mathscr{P}^{\pm}(\nabla \alpha,\nabla\beta)+\mathtt{i}\left( \mathscr{P}^{\pm}(\nabla{\kappa})-\mathscr{P}^{\mp}(\nabla{\kappa})\right) .
	\end{split}
	\end{equation}
\end{lem}
\begin{rem}
	The $8\times 8$ matrix $\mathcal{Q}'_{\mu,\gamma}$ has the form
	\begin{equation}\label{eq:Q'form}
	\mathcal{Q}'_{\mu,\gamma}=\mathrm{diag} (\mathcal{Q}'_{\mu,\gamma}) +2\mathtt{i}\left(
	\begin{array}{cccc}
	0&0&0& 0\\ 0&0& 0& \nabla\kappa\wedge \\
	0& 0&0&0\\ 0&-\nabla\kappa\wedge&0&0
	\end{array}\right).
	\end{equation}
\end{rem}
\begin{proof}[Proof of Lemma~\ref{lem:lem2.3}]

The proof can actually be found in \cite{COS09}. However, for the convenience of readers, we provide a more elementary and straightforward proof. In what follows, for national simplicity, the subscript ${\mu,\gamma}$ of $\mathcal{W}_{\mu,\gamma}$, $\mathcal{W}'_{\mu,\gamma}$ and $\mathcal{Q}'_{\mu,\gamma}$ are omitted.
	
	It is directly verified that
	\begin{equation}
	\kappa \mathscr{P}^{\mp}(\nabla)-\mathscr{P}^{\mp}(\nabla)(\kappa \mathbf{I}_8)=-\mathscr{P}^{\mp}(\nabla\kappa),
	\end{equation}
	where the equality is understood in the operator sense.
	We also claim that
	\begin{equation}
	\begin{split}
	& \mathscr{P}^{\pm}(\nabla \alpha,\nabla\beta)\mathscr{P}^{\mp}(\nabla)-\mathscr{P}^{\mp}(\nabla)\mathscr{P}^{\pm}(\nabla\beta,\nabla \alpha)\\
	=&-\mathscr{P}^{\mp}(\nabla_{\alpha},\nabla_{\beta})\mathscr{P}^{\pm}(\nabla\beta,\nabla \alpha),
	\end{split}
	\end{equation}
	whose verification will be shown right after the current proof of Lemma~\ref{lem:lem2.3}.	
	Hence, we obtain
	\begin{equation}\label{eq:eq3.2600}
	\begin{split}
	&\mathcal{W}\mathscr{P}^{\mp}(\nabla)-\mathscr{P}^{\mp}(\nabla)\mathcal{W}'\\
	=& \mathtt{i}\left[\kappa \mathscr{P}^{\mp}(\nabla)-\mathscr{P}^{\mp}(\nabla)(\kappa \mathbf{I}_8) \right] +\frac{1}{2}\left[ \mathscr{P}^{\pm}(\nabla \alpha,\nabla\beta)\mathscr{P}^{\mp}(\nabla)-\mathscr{P}^{\mp}(\nabla)\mathscr{P}^{\pm}(\nabla\beta,\nabla \alpha)\right] \\
	=&-\frac{1}{2}\mathscr{P}^{\mp}(\nabla_{\alpha},\nabla_{\beta})\mathscr{P}^{\pm}(\nabla\beta,\nabla \alpha)
	-\mathtt{i}\mathscr{P}^{\mp}(\nabla\kappa).
	\end{split}
	\end{equation}
	Moreover, it is computed by noting \eqref{eq:eq2.6} and \eqref{eq:eq4.13} that
	\begin{equation}\label{eq:eq2.23}
	\begin{split}
	\mathcal{W}\mathcal{W}'
	=&\left( \mathtt{i}\kappa \mathbf{I}_8 +\frac{1}{2}\mathscr{P}^{\pm}(\nabla \alpha,\nabla\beta)\right)\left(\mathtt{i}\kappa \mathbf{I}_8 +\frac{1}{2}\mathscr{P}^{\pm}(\nabla\beta,\nabla \alpha) \right)  \\
	=&-\kappa^2 \mathbf{I}_8+\frac{1}{4}\mathrm{diag}(\nabla \alpha\cdot \nabla \alpha, \nabla \beta\cdot \nabla \beta)
	+ \mathtt{i}\mathscr{P}^{\pm}(\nabla{\kappa}) .
	\end{split}
	\end{equation}
	The identity \eqref{eq:eq2.20} is then shown by \eqref{eq:eq3.2600}, \eqref{eq:eq2.23} and \eqref{eq:eq3.5}.
	Similarly, we have
	\begin{equation}
	\begin{split}
	\mathcal{W}'\mathcal{W}
	=&-\kappa^2 \mathbf{I}_8+\frac{1}{4}\mathrm{diag}(\nabla \beta\cdot \nabla \beta,\nabla \alpha\cdot \nabla \alpha)+\mathtt{i}\mathscr{P}^{\pm}(\nabla{\kappa}).
	\end{split}
	\end{equation}
	Thus, \eqref{eq:eq2.22} can be seen from that
	\begin{equation}
	\begin{split}
	\mathcal{Q}'=&\mathcal{W}'\mathcal{W}-\left(\mathscr{P}^{\mp}(\nabla){W}-{W}'\mathscr{P}^{\mp}(\nabla) \right)\\
	=& \mathcal{W}'\mathcal{W}+\mathtt{i}\left[\kappa \mathscr{P}^{\mp}(\nabla)-\mathscr{P}^{\mp}(\nabla)(\kappa \mathbf{I}_8) \right]\\
	&+\frac{1}{2}\left[ \mathscr{P}^{\pm}(\nabla\beta,\nabla \alpha)\mathscr{P}^{\mp}(\nabla)-\mathscr{P}^{\mp}(\nabla)\mathscr{P}^{\pm}(\nabla \alpha,\nabla\beta)\right] \\
	=& \mathcal{W}'\mathcal{W}-\mathtt{i}\mathscr{P}^{\mp}(\nabla\kappa)-\frac{1}{2}\mathscr{P}^{\mp}(\nabla_{\beta},\nabla_{\alpha})\mathscr{P}^{\pm}(\nabla \alpha,\nabla\beta).
	\end{split}
	\end{equation}
\end{proof}

\begin{proof}[Proof of the equation \eqref{eq:eq3.2600}]
	Let $\mathcal{Z}$ be any 8-dimensional vector of the form $\mathcal{Z}=(\phi,\mathbf{H}^T,\psi,\mathbf{E}^T)^T$. Then
	\begin{equation}
	\mathscr{P}^{\mp}(\nabla)\mathcal{Z}=\left(
	\begin{array}{c}
	\nabla \cdot \mathbf{E}\\ \nabla \psi -\nabla\wedge \mathbf{E}\\ \nabla \cdot \mathbf{H}\\ \nabla \phi +\nabla\wedge \mathbf{H}
	\end{array}\right),
	\end{equation}
	and hence
	\begin{equation}
	\begin{split}
	\mathscr{P}^{\pm}(\nabla \alpha,\nabla\beta)\mathscr{P}^{\mp}(\nabla)\mathcal{Z}
	=& \left(
	\begin{array}{cccc}
	0&0&0& \nabla\alpha\cdot\\
	0&0& \nabla\alpha& \nabla\alpha\wedge \\
	0& \nabla\beta\cdot&0&0\\
	\nabla\beta&-\nabla\beta\wedge&0&0
	\end{array}\right)\left(
	\begin{array}{c}
	\nabla \cdot \mathbf{E}\\ \nabla \psi -\nabla\wedge \mathbf{E}\\ \nabla \cdot \mathbf{H}\\ \nabla \phi +\nabla\wedge \mathbf{H}
	\end{array}\right)\\
	=& D\mathcal{Z}+\left(
	\begin{array}{c}
	\nabla\alpha\cdot\nabla\phi\\ \nabla\alpha\wedge\nabla\phi\\
	\nabla\beta\cdot\nabla\psi\\ -\nabla\beta\wedge\nabla\psi
	\end{array}\right),
	\end{split}
	\end{equation}
	with
	\begin{equation}
	D\mathcal{Z}=\left(
	\begin{array}{c}
	\nabla\alpha\cdot (\nabla\wedge \mathbf{H})\\
	(\nabla\cdot \mathbf{H})\nabla\alpha+\nabla_\mathbf{H} ( \mathbf{H}\cdot\nabla\alpha)-(\nabla\alpha\cdot\nabla)\mathbf{H}\\
	-\nabla\beta\cdot (\nabla\wedge \mathbf{E})\\
	(\nabla\cdot \mathbf{E})\nabla\beta+\nabla_\mathbf{E} ( \mathbf{E}\cdot\nabla\beta)-(\nabla\beta\cdot\nabla)\mathbf{E}
	\end{array}\right),
	\end{equation}
	by noting for any 3-dimensional fields $\mathbf{v}$ and $\mathbf{F}$ that
	\begin{equation}
	\mathbf{v}\wedge(\nabla\wedge \mathbf{F})=\nabla_\mathbf{F}(\mathbf{v}\cdot \mathbf{F})-(\mathbf{v}\cdot\nabla) \mathbf{F},
	\end{equation}
	where the subscript $_\mathbf{F}$ in $\nabla_\mathbf{F}$ means that the differential operator acts only on $\mathbf{F}$.
	Similarly, we have that
	\begin{equation}
	\mathscr{P}^{\pm}(\nabla\beta,\nabla \alpha)\mathcal{Z}=\left(
	\begin{array}{c}
	\mathbf{E} \cdot \nabla\beta\\ \psi\nabla\beta-\mathbf{E}\wedge\nabla \beta\\
	\mathbf{H} \cdot \nabla\alpha\\ \phi\nabla\alpha+\mathbf{H}\wedge\nabla \alpha
	\end{array}\right),
	\end{equation}
	and thus by applying the identities
	\begin{equation}
	\nabla(\mathbf{F} \cdot\nabla v)=(\mathbf{F} \cdot \nabla)\nabla v+\nabla_\mathbf{F}(\mathbf{F} \cdot\nabla v),
	\end{equation}
	\begin{equation}
	\nabla\wedge(v \mathbf{F})=\nabla v\wedge \mathbf{F}+ v\nabla\wedge \mathbf{F},
	\end{equation}
	\begin{equation}
	\nabla\cdot(\mathbf{F} \wedge\mathbf{v})=\mathbf{v}\cdot(\nabla\wedge \mathbf{F})-\mathbf{F} \cdot(\nabla\wedge \mathbf{v}),
	\end{equation}
	and
	\begin{equation}
	\nabla\wedge(\mathbf{F}\wedge \mathbf{v})=(\nabla\cdot \mathbf{v})\mathbf{F}+(\mathbf{v}\cdot\nabla) \mathbf{F}-(\nabla\cdot \mathbf{F})\mathbf{v}-(\mathbf{F} \cdot\nabla)\mathbf{v},
	\end{equation}
	that
	\begin{equation}
	\begin{split}
	\mathscr{P}^{\mp}(\nabla)\mathscr{P}^{\pm}(\nabla\beta,\nabla \alpha)\mathcal{Z}
	=& D\mathcal{Z}+\left(
	\begin{array}{c}
	\nabla\alpha\cdot\nabla\phi+\phi\Delta\alpha\\ 2(\mathbf{H}\cdot\nabla)\nabla\alpha+\nabla\alpha\wedge\nabla\phi-\mathbf{H} \Delta\alpha\\
	\nabla\beta\cdot\nabla\psi+\psi\Delta\beta\\ 2(\mathbf{E}\cdot\nabla)\nabla\beta-\nabla\beta\wedge\nabla\psi-\mathbf{E} \Delta\beta
	\end{array}\right).
	\end{split}
	\end{equation}
	Therefore
	\begin{equation}
	\begin{split}		
	&\left( \mathscr{P}^{\pm}(\nabla \alpha,\nabla\beta)\mathscr{P}^{\mp}(\nabla)-\mathscr{P}^{\mp}(\nabla)\mathscr{P}^{\pm}(\nabla\beta,\nabla \alpha)\right) \mathcal{Z}\\
	&=	\left(
	\begin{array}{c}
	-\phi\Delta\alpha\\ \mathbf{H}\Delta\alpha-2(\mathbf{H}\cdot\nabla)\nabla\alpha\\
	-\psi\Delta\beta\\ \mathbf{E}\Delta\beta-2(\mathbf{E}\cdot\nabla)\nabla\beta
	\end{array}\right)\\
	&=-\left(\mathscr{P}^{\mp}(\nabla_{\alpha},\nabla_{\beta})\mathscr{P}^{\pm}(\nabla\beta,\nabla \alpha) \right) \mathcal{Z}	.
	\end{split}
	\end{equation}
\end{proof}

\subsection{Results on the $8$-dimensional Sch\"{o}rdinger system \eqref{eq:Zeqn}}\label{sec:3.2}

We first extend the definition of the functions $\mu$ and $\gamma$ in Theorem~\ref{thm:thm2.1} from $\Omega$ into the whole space $\mathbb{R}^3$, as guaranteed by the following lemma (cf.\cite[Theorem 5.24]{AdF03}).
\begin{lem}
	Given $\Omega\subset\mathbb{R}^3$ a domain satisfying the strong local Lipschitz condition, there exists a bounded linear extension operator
	\begin{equation}
	\mathscr{E}: \mathscr{W}_p(\Omega)\rightarrow \mathscr{W}_p(\mathbb{R}^3),
	\end{equation}
	for every $p\in\mathbb{R}$, $1\leq p<\infty$, and every integer $m$, such that
	\begin{equation}
	\mathscr{E} u=u\quad\mbox{a.e. in $\Omega$},\qquad \forall u\in \mathscr{W}_p(\Omega).
	\end{equation}
\end{lem}

Moreover, we assume, possibly after multiplying a $C^{\infty}_c(\mathbb{R}^3)$ function which equals to $1$ in $\Omega$, that $\mathscr{E}u$ is compactly supported in $\mathbb{R}^3$.
For notational simplicity, we still write the extended $\mathscr{E}\mu$ and $\mathscr{E}\gamma$ as $\mu$ and $\gamma$, respectively.
In addition, all the fields used in Section~\ref{sec:martixForm}  associated with $\mu$ and $\gamma$,  such as $\mathcal{Q}_{\mu,\gamma}$ and $\mathcal{W}_{\mu,\gamma}$, are correspondingly compactly extended to $\mathbb{R}^3\backslash \Omega$.

Given $s\in\mathbb{R}$ and $1\le q\le \infty$, we shall make use of the \emph{generalized Sobolev space} $\mathscr{H}^{s}_{q}$ defined by (cf. \cite{BeL76})
\begin{equation}
\mathscr{H}^{s}_{q}=\{f\in \mathscr{S}';\|f\|_{\mathscr{H}^{s}_{q}}<\infty\},
\end{equation}
with the norm
\begin{equation}
\|f\|_{\mathscr{H}^{s}_{q}}=\|\mathscr{F}^{-1}\{(1+|\cdot|^2)^{s/2}\mathscr{F}f\}\|_{L_q},
\end{equation}
where $\mathscr{F}$ and $\mathscr{F}^{-1}$ denotes the Fourier transform and its inverse operator, respectively.

Recall that for $m\in\mathbb{N}$ a positive integer, $\mathscr{H}^{m}_{q}$ is equivalent to the normal Sobolev space $\mathscr{W}^{m}_{q}$, namely,
\begin{equation}
\mathscr{H}^{m}_{q}=\mathscr{W}^{m}_{q},\quad m\in\mathbb{N}.
\end{equation}
Moreover, one has in $\mathbb{R}^n$ the following embedding theorem (cf. \cite[Theorem 6.5.1]{BeL76})
\begin{equation}\label{eq:embed}
\mathscr{H}^{s}_{q}\subset \mathscr{H}^{s_1}_{q_1},
\end{equation}
where $1<q< q_1<\infty$ and $s,s_1\in \mathbb{R}$ are such that
\begin{equation}
s-n/q=s_1-n/q_1.
\end{equation}

The main contribution of the current subsection is the following result.
\begin{prop}\label{prop:CGO_Z}
	Let $s\in [0,2]$, $q\in[4,6)$, $\mathcal{Z}_0\in\mathbb{C}^8$ and $\bm{\zeta}\in \mathbb{C}^3$ with $\bm{\zeta}\cdot\bm{\zeta}=0$.
	Suppose that $|\bm{\zeta}|$ is sufficiently large. Then there exists a CGO 　solution to
	\begin{equation}\label{eq:ZeqnR}
	\left(-\Delta \mathbf{I}_8+\mathcal{Q}_{\mu,\gamma} \right) \mathcal{Z}=0,
	\end{equation}
	which is of the form
	\begin{equation}\label{eq:Zform}
	\mathcal{Z}=e^{-\bm{\zeta}\cdot \mathbf{x}}\left( \mathcal{Z}_0+ \tilde{Z}_{\bm{\zeta},\mathcal{Z}_0} \right),
	\end{equation}
	such that
	\begin{equation}
	\|\tilde{\mathcal{Z}}\|_{\mathscr{H}^{s}_{q}}\le \frac{C|\mathcal{Z}_0|}{|\bm{\zeta}|^{6/q-1}}\|\mathcal{Q}_{\mu,\gamma}\hat{\mathcal{Z}}_0\|_{\mathscr{H}^{s}_{q'}}.
	\end{equation}
\end{prop}

To prove Proposition~\ref{prop:CGO_Z}, the following lemma extracted from \cite[Proposition 3.3]{PSV14arXiv} is of important use.
\begin{lem}\label{lem:est}
	Let $s\in\mathbb{R}$, $q\in[4,6]$ and $\bm{\zeta}\in \mathbb{C}^3\backslash\mathbb{R}^3$.
	Then for any $f\in \mathscr{H}^{s}_{q'}(\mathbb{R}^3)$, there exists a solution $\psi\in  \mathscr{H}^{s}_{q}(\mathbb{R}^3)$ to
	\begin{equation}\label{eq:CGO1}
	\left( \Delta+2\bm{\zeta}\cdot \nabla \right)\psi=f,
	\end{equation}
	which satisfies
	\begin{equation}
	\|\psi\|_{\mathscr{H}^{s}_{q}}
	\le \frac{C}{|\Im\bm{\zeta}|^{6/q-1}}\|f\|_{\mathscr{H}^{s}_{q'}}.
	\end{equation}
\end{lem}

\begin{lem}\label{lem:Qtilde}
	Let $s$ and $q$ be real numbers such that $s>0$ and $q\ge 2$. Denote $\tilde{q}:=q/(q-2)$. Given an $8\times 8$-matrix field $\mathcal{Q}$ with $\mathscr{H}^s_{\tilde{q}}$ regularity, define the multiplier $\mathscr{M}_{\mathcal{Q}}$ on any $8$-dimensional vector field $\mathcal{F}$ by
	\begin{equation}
	\mathscr{M}_{\mathcal{Q}}\mathcal{F}:=\mathcal{Q}\mathcal{F}.
	\end{equation}
	Then $\mathscr{M}_{\mathcal{Q}}$ is continuous from $\mathscr{H}^{\tilde{s}}_{q}$ to $\mathscr{H}^{\tilde{s}}_{q'}$ for any $\tilde{s}\in [0,s]$.
%
\end{lem}	

\begin{proof}
	The proof follows the idea in the proof of \cite[Proposition 3.5]{PSV14arXiv}, which shows the similar result but for the scalar case with $s\in [0,1]$.
	
	Define the operator $\mathscr{M}$ by $$\mathscr{M}(\tilde{\mathcal{Q}},\mathcal{F}):=\tilde{\mathcal{Q}}\mathcal{F}$$ for any $8\times 8$-matrix field $\tilde{\mathcal{Q}}$ and any $8$-dimensional vector field $\mathcal{F}$. Then
	\begin{equation}\label{eq:eq4.56}
	\mathscr{M}(\mathcal{Q},\mathcal{F})=\mathscr{M}_{\mathcal{Q}}\mathcal{F}.
	\end{equation}
	Let $n_0:=[s]+1\ge 1$. We claim that
	$$\mathscr{M}: \mathscr{H}^n_{\tilde{q}}\times \mathscr{H}^{n}_{q} \rightarrow \mathscr{H}^{n}_{q'},$$
	is continuous for any $n\in\N_0$ such that $n\le n_0$.
	In fact, the case for $n=0$ is immediately implied by the following inequality given in \cite[Page 13]{PSV14arXiv},
	\begin{equation*}
	\|\tilde{\mathcal{Q}}\mathcal{F}\|_{L_{q'}}
	\le \|\tilde{\mathcal{Q}}\|_{L_{\tilde{q}}}\|\mathcal{F}\|_{L_{q}}.
	\end{equation*}
	The assertion for any $n\le n_0$ can be then verified by induction with the simple relation
	$$\nabla (f_1  f_2) = f_2 \nabla f_1+f_1\nabla f_2.$$
	By \cite[Theorem 4.4.1]{BeL76}, a result for the interpolation spaces, one further has 
	$$\mathscr{M}: \mathscr{H}^{\tilde{s}}_{\tilde{q}}\times \mathscr{H}^{\tilde{s}}_{q} \rightarrow \mathscr{H}^{\tilde{s}}_{q'},$$
	is also continuous for any $\tilde{s}\in\R$ with $0\le \tilde{s}\le n_0$.
	The proof can then be completed by recalling the relation \eqref{eq:eq4.56}.
\end{proof}

\begin{lem}\label{lem:lem3.6}
	Let the domain $\Omega\subset\mathbb{R}^3$ be bounded and satisfy the strong local Lipschitz condition, and let $\mu,\gamma\in C^{m+1}(\Omega)$. Then the extended $8\times 8$ matrices $\mathcal{Q}_{\mu,\gamma}$ and $\mathcal{W}'_{\mu,\gamma}$ are in $\mathscr{H}^{s}_{p}$ with $0\le s\le m$ and $1<p<\infty$.
\end{lem}
\begin{proof}
	It is an application of the embedding \eqref{eq:embed}.
\end{proof}

\begin{proof}[Proof of Proposition~\ref{prop:CGO_Z}]
	For $\mathcal{Z}$ given in \eqref{eq:Zform} solving \eqref{eq:ZeqnR}, it can be derived that $\tilde{\mathcal{Z}}=\tilde{\mathcal{Z}}_{\bm{\zeta},\mathcal{Z}_0}$ satisfies
	\begin{equation}\label{eq:Ztilde}
	\left(\Delta+2\bm{\zeta}\cdot\nabla \right) \tilde{\mathcal{Z}}
	= \mathcal{Q}_{\mu,\gamma} \left( \mathcal{Z}_0 + \tilde{\mathcal{Z}} \right).
	\end{equation}
	Indeed, denoting by $\mathscr{G}_{\bm{\zeta}}$ the solution operator of \eqref{eq:CGO1} introduced in Lemma~\ref{lem:est},
	the equation \eqref{eq:Ztilde} can be achieved by
	\begin{equation}
	(\mathbf{I}_8-\mathscr{G}_{\bm{\zeta}} \mathcal{Q}_{\mu,\gamma})\tilde{\mathcal{Z}}=\mathscr{G}_{\bm{\zeta}} \mathcal{Q}_{\mu,\gamma}\mathcal{Z}_0.
	\end{equation}
	Notice by Lemmas~\ref{lem:est} and {\ref{lem:Qtilde}} that,
	\begin{equation}
	\|\mathscr{G}_{\bm{\zeta}} \mathcal{Q}_{\mu,\gamma}\mathcal{F}\|_{\mathscr{H}^{s}_{q}}\le \frac{C}{|\bm{\zeta}|^{6/q-1}}\| \mathcal{Q}_{\mu,\gamma}\mathcal{F}\|_{\mathscr{H}^{s}_{q'}}
	\le \frac{C}{|\bm{\zeta}|^{6/q-1}}\|\mathcal{F}\|_{\mathscr{H}^{s}_{q}}
	\end{equation}
	holds for any 8-dimensional fields $\mathcal{F}$ with $\mathscr{H}^{s}_{q}$ regularity.
	Hence by taking $\bm{\zeta}\in \mathbb{C}^3$ with $|\bm{\zeta}|^{6/q-1}>C$, we can obtain by Neumann series that
	\begin{equation}
	\|\tilde{\mathcal{Z}}\|_{\mathscr{H}^{s}_{q}}
	\le  \frac{C_q|\mathcal{Z}_0|}{|\bm{\zeta}|^{6/q-1}}\|\mathcal{Q}_{\mu,\gamma}\hat{\mathcal{Z}}_0\|_{\mathscr{H}^{s}_{q'}}.
	\end{equation}
\end{proof}

The following result is a consequence of Proposition~\ref{prop:CGO_Z}  
by taking specific values of $\mathcal{Z}_0$.
\begin{cor}\label{cor:cor3.2}
	Let $\bm{\zeta}$, $q$ and $s$ be the same as in Proposition~\ref{prop:CGO_Z}, and let $\bm{\eta}\in\C^3$ be such that $\bm{\eta}\cdot \bm{\zeta}=0$.
	Given four constants $c_j^{\mathbf{E}},c_j^{\mathbf{H}}\in\mathbb{R}$, $j=1,2$, set
	\begin{equation}\label{eq:Z0}
	\mathcal{Z}_0= -\frac{1}{|\bm{\zeta}|}
	\left( c_1^{\mathbf{E}},\,-c_2^{\mathbf{E}}\bm{\eta}^T,\,c_1^{\mathbf{H}},\, c_2^{\mathbf{H}}\bm{\eta}^T\right)^T.
	\end{equation}
	Then
	\begin{equation}\label{eq:PZ0}
	-\mathscr{P}^{\mp}(\bm{\zeta})\mathcal{Z}_0=
	\left( 0,\,c_1^{\mathbf{H}} \hat{\bm{\zeta}}^T+c_2^{\mathbf{H}}(\bm{\eta}\wedge\hat{\bm{\zeta}})^T,\,0, \,c_1^{\mathbf{E}} \hat{\bm{\zeta}}^T+c_2^{\mathbf{E}}(\bm{\eta}\wedge\hat{\bm{\zeta}})^T\right)^T.
	\end{equation}
	Moreover, let the constants $c_j^{\mathbf{E}},c_j^{\mathbf{H}}$, $j=1,2$, be independent of $|\bm{\zeta}|$. Then there exists a CGO 　solution to \eqref{eq:ZeqnR}
	of the form \eqref{eq:Zform}
	such that
	\begin{equation}
	\|\tilde{\mathcal{Z}}\|_{\mathscr{H}^{s}_{q}}\le \frac{C}{|\bm{\zeta}|^{6/q}}\|\mathcal{Q}_{\mu,\gamma}\hat{\mathcal{Z}}_0\|_{\mathscr{H}^{s}_{q'}}.
	\end{equation}
\end{cor}

\subsection{Proof of Theorem~\ref{thm:thm2.1}}

\begin{prop}\label{prop:Z0require}
	Let $\bm{\zeta}$, $\mathcal{Z}$, $\tilde{\mathcal{Z}}$ and $\mathcal{Z}_0$ be the same as in Proposition~\ref{prop:CGO_Z}.
	Define the $8$-dimensional field $\mathcal{Y}$ by
	\begin{equation}\label{eq:YdefnZ}
	\mathcal{Y}=\left(\mathscr{P}^{\mp}(\nabla)-\mathcal{W}'_{\mu,\gamma}\right)\mathcal{Z}.
	\end{equation}
	Then $\mathcal{Y}$ has the form
	\begin{equation}\label{eq:Y}
	\mathcal{Y}=e^{-\bm{\zeta}\cdot \mathbf{x}}\left( \mathcal{Y}_0+ \tilde{\mathcal{Y}} \right),
	\end{equation}
	with
	\begin{equation}\label{eq:Y0}
	\mathcal{Y}_0=-\mathscr{P}^{\mp}(\bm{\zeta})\mathcal{Z}_0,
	\end{equation}
	and
		\begin{equation}\label{eq:Ytilde}
		\tilde{\mathcal{Y}}= \mathscr{P}^{\mp}(\nabla)\tilde{\mathcal{Z}}-\mathscr{P}^{\mp}(\bm{\zeta})\tilde{\mathcal{Z}}-\mathcal{W}'_{\mu,\gamma}(\tilde{\mathcal{Z}}+\mathcal{Z}_0).
		\end{equation}
	Moreover, for sufficiently large  $|\bm{\zeta}|$, one has
	\begin{equation}
	\mathcal{Y}^{h}=\mathcal{Y}^{e}=0,
	\end{equation}
	provided
	\begin{equation}\label{eq:Z0req}
	\left( \mathscr{P}^{\mp}(\bm{\zeta})\mathcal{Z}_0\right)^{h}=\left( \mathscr{P}^{\mp}(\bm{\zeta}) \mathcal{Z}_0\right)^{e}=0,
	\end{equation}		
	where the notations are similar as the ones in \eqref{eq:calX}.
\end{prop}
\begin{proof}
	Inserting the form \eqref{eq:Zform} of $\mathcal{Z}$ into the definition \eqref{eq:YdefnZ} of $\mathcal{Y}$, it is directly obtained that $\mathcal{Y}$ has the form \eqref{eq:Y}-\eqref{eq:Ytilde}.
		
	Notice by the identities \eqref{eq:ZeqnR} and \eqref{eq:eq2.20} that $\mathcal{Y}$ is a solution to
	\begin{equation}
		\left(\mathscr{P}^{\mp}(\nabla)+\mathcal{W}_{\mu,\gamma}\right)\mathcal{Y}=0,
	\end{equation}
	and further by \eqref{eq:eq2.22} that $\mathcal{Y}$ solves
	\begin{equation}
	\left(\Delta \mathbf{I}_8-\mathcal{Q}'_{\mu,\gamma}\right) \mathcal{Y}=0.
	\end{equation}
	Hence by recalling \eqref{eq:Q'form}, the first and the fifth components of $\mathcal{Y}$ satisfy
	\begin{equation}\label{eq:eq4.69}
	(-\Delta +q_{\beta})\mathcal{Y}^{h}=(-\Delta +q_{\alpha})\mathcal{Y}^{e}=0,
	\end{equation}
	where the compactly supported potentials $q_\alpha$ and $q_\beta$ are given by
	\begin{equation}
	q_\tau=\frac{1}{4}\nabla \tau\cdot \nabla \tau-\frac{1}{2}\Delta \tau-\kappa^2,\quad \tau=\alpha,\beta.
	\end{equation}	
	Then it is observed by \eqref{eq:Z0req} that $\mathcal{Y}_0^{h}=\mathcal{Y}_0^{e}=0$, and consequently that,
	\begin{equation}\label{eq:eq3.63}
	\mathcal{Y}^{\tau}=e^{-\bm{\zeta}\cdot \mathbf{x}}\left( \mathcal{Y}_0^{\tau}+ \tilde{\mathcal{Y}}^{\tau} \right)=e^{-\bm{\zeta}\cdot \mathbf{x}}\tilde{\mathcal{Y}}^{\tau},\quad \tau=h,e.
	\end{equation}
	Further applying \eqref{eq:eq4.69}, $\tilde{\mathcal{Y}}^{h}$ solves
	\begin{equation}
	(\Delta+2\bm{\zeta}\cdot\nabla)\tilde{\mathcal{Y}}^{h}=q_\beta \tilde{\mathcal{Y}}^{h}.
	\end{equation}
	Recalling the solution operator $\mathscr{G}_{\bm{\zeta}}$ of \eqref{eq:CGO1}, the equation for $\tilde{\mathcal{Y}}^{h}$ can be further expressed by 
	\begin{equation}
	(\mathbf{I}_8-\mathscr{G}_{\bm{\zeta}}q_\beta)\tilde{\mathcal{Y}}^{h}= \tilde{\mathcal{Y}}^{h}.
	\end{equation}
	It can be obtained by similar arguments as in the proof of Proposition~\ref{prop:CGO_Z} that the operator $(\mathbf{I}_8-\mathscr{G}_{\bm{\zeta}}q_\beta)$ is invertible for large $|\bm{\zeta}|$. Therefore we deduce that $\tilde{\mathcal{Y}}^{h}=0$ and hence $\mathcal{Y}^{h}=0$.
	The assertion $\mathcal{Y}^{e}=0$ can be shown analogously.
	
	The proof is complete.
\end{proof}

\begin{rem}
	Notice that
	\begin{equation}
	\mathscr{P}^{\mp}(\bm{\zeta})\mathcal{Z}_0=
	\left(
	\begin{array}{c}
	\bm{\zeta} \cdot \mathcal{Z}_0^\mathbf{E}\\
	\mathcal{Z}_0^{e}\bm{\zeta}-\bm{\zeta}\wedge \mathcal{Z}_0^\mathbf{E}\\
	\bm{\zeta} \cdot \mathcal{Z}_0^\mathbf{H}\\
	\mathcal{Z}_0^{h}\bm{\zeta}+\bm{\zeta}\wedge \mathcal{Z}_0^\mathbf{H}
	\end{array}
	\right).
	\end{equation}
	Then the condition \eqref{eq:Z0req} reads
	\begin{equation}\label{eq:Z0reqEH}
	\bm{\zeta} \cdot \mathcal{Z}_0^\mathbf{E}=\bm{\zeta} \cdot \mathcal{Z}_0^\mathbf{H}=0.
	\end{equation}
\end{rem}

\begin{cor}\label{cor:cor3.02}
	Let $\bm{\zeta}$, $\mathcal{Z}$ and $\mathcal{Z}_0$ be the same as in Proposition~\ref{prop:CGO_Z}, and let $\mathcal{Y}$ be the one defined by \eqref{eq:YdefnZ}.
	Define
	\begin{equation}\label{eq:XdefnY}
	\mathcal{X}:=\mathrm{diag}\left(\mu^{-1/2},\gamma^{-1/2}\right) \mathcal{Y}\quad \mbox{in $\Omega$}.
	\end{equation}
	Suppose that $\mathcal{Z}_0$ satisfies \eqref{eq:Z0reqEH}. Then $\mathcal{X}^{h}=\mathcal{X}^{e}=0$, and further by Lemma~\ref{lem:lem2.1}, $\mathbf{E}:=\mathcal{X}^{\mathbf{E}}$ and $\mathbf{H}:=\mathcal{X}^{\mathbf{H}}$ satisfy the Maxwell equations \eqref{eq:MaxwellEH} in $\Omega$. 
\end{cor}

We are now ready to prove Theorem~\ref{thm:thm2.1}.
\begin{proof}[The proof of Theorem~\ref{thm:thm2.1}]
		
	Let $\mathcal{Z}_0$ be the constant 8-dimensional vector defined in \eqref{eq:Z0}. Then $\mathcal{Z}_0$ clearly satisfies \eqref{eq:Z0reqEH}.
	Hence by Corollary~\ref{cor:cor3.02}, $\mathbf{E}:=\gamma^{-1/2}\mathcal{Y}^{E}$ and $\mathbf{H}:=\mu^{-1/2}\mathcal{Y}^{H}$ solve the Maxwell equations \eqref{eq:MaxwellEH_full}, with the 8-dimensional field $\mathcal{Y}$ given by \eqref{eq:Y}-\eqref{eq:Ytilde}.
	More precisely, $\mathbf{E}$ and $\mathbf{H}$ have the form \eqref{eq:EHform} with
	\begin{equation}
	\mathbf{E}_0=\mathcal{Y}_0^{\mathbf{E}},\quad \mathbf{H}_0=\mathcal{Y}_0^{\mathbf{H}},
	\end{equation}
	and
	\begin{equation}
	\tilde{\mathbf{E}}_{\bm{\zeta},\mathbf{E}_0}=\gamma^{-1/2}\tilde{\mathcal{Y}}^{\mathbf{E}},\quad \tilde{\mathbf{H}}_{\bm{\zeta},\mathbf{H}_0}=\mu^{-1/2}\tilde{\mathcal{Y}}^{\mathbf{H}}.
	\end{equation}
	Then \eqref{eq:EH0} is obtained by recalling \eqref{eq:Y0} and \eqref{eq:PZ0}.
	
	It remains to verify the estimates in \eqref{eq:eq2.4}.
	Taking $q\in[4,6)$ and $s\in(0,3/q)$, Corollary~\ref{cor:cor3.2} implies
	\begin{equation}
	\|\mathscr{P}^{\mp}(\nabla)\tilde{\mathcal{Z}}\|_{\mathscr{H}^{s}_{q}}\le \frac{C}{|\bm{\zeta}|^{6/q}}\|\mathcal{Q}_{\mu,\gamma}\hat{\mathcal{Z}}_0\|_{\mathscr{H}^{s+1}_{q'}},
	\end{equation}
	and
	\begin{equation}
	\|\mathscr{P}^{\mp}(\bm{\zeta})\tilde{\mathcal{Z}}\|_{\mathscr{H}^{s}_{q}}\le \frac{C}{|\bm{\zeta}|^{6/q-1}}\|\mathcal{Q}_{\mu,\gamma}\hat{\mathcal{Z}}_0\|_{\mathscr{H}^{s}_{q'}}.
	\end{equation}
	Moreover, recalling by Lemma~\ref{lem:lem3.6} that $\mathcal{W}'_{\mu,\gamma}\in \mathscr{H}^2_q$, we have
		\begin{equation}
		\|\mathcal{W}'_{\mu,\gamma}\tilde{\mathcal{Z}}\|_{\mathscr{H}^{s}_{q}}\le {C_W}\|\tilde{\mathcal{Z}}\|_{\mathscr{H}^{s}_{q}}\le \frac{C}{|\bm{\zeta}|^{6/q}},
		\end{equation}
	and
	\begin{equation}
	\|\mathcal{W}'_{\mu,\gamma}\mathcal{Z}_0\|_{\mathscr{H}^{s}_{q}}=\frac{C}{|\bm{\zeta}|}\|\mathcal{W}'_{\mu,\gamma}\hat{\mathcal{Z}}_0\|_{\mathscr{H}^{s}_{q}}\le \frac{C}{|\bm{\zeta}|} .
	\end{equation}	
	Therefore one has by \eqref{eq:Ytilde} that
	\begin{equation}
	\|\tilde{\mathcal{Y}}\|_{\mathscr{H}^{s}_{q}}\le \frac{C}{|\bm{\zeta}|^{6/q-1}},
	\end{equation}
	with $|\bm{\zeta}|$ sufficiently large.
	Then the Sobolev embedding $\mathscr{H}^s_{q}\subset L_p$ for $$3/p=3/q-s$$ in dimension three yields,
	\begin{equation}\label{eq:eq3.36}
	\|\tilde{\mathcal{Y}}\|_{L_{p}}
	\le \frac{C}{|\bm{\zeta}|^{3/p+(3/q+s-1)}}.
	\end{equation}	
	For any $t\in(0,1)$, set
	\begin{equation}
	s={3}/{q}-t\left({6}/{q}-1\right),
	\end{equation}
	then
	\begin{equation}
	p=\frac{3}{t\left({6}/{q}-1\right)}
	\end{equation}
	can take any number in $(6,\infty)$ by proper choices of values for $t\in(0,1)$ and $q\in[4,6)$.
	Hence the relation \eqref{eq:eq3.36} can be reformulated as
	\begin{equation}
	\|\tilde{\mathcal{Y}}\|_{L_{p}}
	\le \frac{C}{|\bm{\zeta}|^{3/p+\delta}},
	\end{equation}
	with
	\begin{equation}
	\delta=3/q+s-1=(1-t)\left({6}/{q}-1\right)>0.
	\end{equation}		
	Finally, the estimates in \eqref{eq:eq2.4} can be achieved by noting the inequality
	\begin{equation}
	\|\tilde{\mathbf{E}}_{\bm{\zeta},\mathbf{E}_0}\|_{L_p}\leq C\|\gamma\|_{L_{\infty}}\|\tilde{\mathcal{Y}}^{\mathbf{E}}\|_{L_p}
	\end{equation}
	for $\tilde{\mathbf{E}}_{\bm{\zeta},\mathbf{E}_0}$ and the analogous one for $\tilde{\mathbf{H}}_{\bm{\zeta},\mathbf{H}_0}$.
	
	The proof is complete.
\end{proof}

\section{Non-vanishment of the Laplace transform}\label{sec:LapTrans}

\subsection{Preliminaries}
We first fix some notations. Let $\bm{\eta}\in\mathbb{C}^3$ and let $\bm{\alpha}=(\alpha_j)_{j=1}^3$ be a $3$-dimensional multi-index.
The factorial $\bm{\alpha}!$ is given by
\begin{equation}
\bm{\alpha}!=\sum_{\substack{j=1\\\alpha^{(j)}\ne 0}}^{3} \alpha^{(j)}!.
\end{equation}
Recall the scalar $\bm{\eta}^{\bm{\alpha}}$ is defined in \eqref{eq:indexMulti}.
Particularly,  $\bm{\eta}^{\mathbf{1}}$ and $\bm{\eta}^{\mathbf{2}}$ are respectively the products
\begin{equation}\label{eq:eq0.4}
\bm{\eta}^{\mathbf{1}}=\prod_{j=1}^3\eta^{(j)}\quad\text{and}\quad \bm{\eta}^{\mathbf{2}}=\prod_{j=1}^3\left( \eta^{(j)}\right) ^2.
\end{equation}
The following definition of $\frac{1}{\bm{\eta}}$ shall also be needed,
\begin{equation}
\frac{1}{\bm{\eta}}:=\left(\frac{1}{\eta^{(j)}}\right)_{j=1}^3.
\end{equation}
We define
\begin{equation}
\frac{\mathbf{x}}{\bm{\eta}}:=\left(\frac{x^{(j)}}{\eta^{(j)}}\right)_{j=1}^3,
\end{equation}
for any $3$-dimensional fields $\mathbf{x}=(x^{(j)})_{j=1}^3$.
Given $j=1,2,3$, $\mathbf{x}_{\hat{j}}$ signifies the $3$-dimensional vector $\mathbf{x}$ with the $j$-th Cartesian component $x^{(j)}=0$.
%

Let $\mathbf{P}_N$ be a $3$-dimensional homogeneous polynomial of degree $N$. Then $\mathbf{P}_N=(P_N^{(j)})_{j=1}^3$ can be represented as
\begin{equation}
P_N^{(j)}(\mathbf{x})=\sum_{|\bm{\alpha}|=N}p_{\bm{\alpha}}^{(j)} \mathbf{x}^{\bm{\alpha}},\quad j=1,2,3.
\end{equation}
By straightforward (though a bit tedious) calculations, one can show the following proposition.

\begin{prop}\label{prop:rls1}
	Some relations on the coefficients of $\mathbf{P}_N$ are in order.
	\begin{enumerate}
		\item[i)]
		If the polynomial is divergence-free, namely
		\begin{equation}\label{eq:divfree}
		\nabla\cdot \mathbf{P}_N\equiv 0,
		\end{equation}
		then one has
		\begin{equation}
		\begin{split}
		&\sum_{j=1}^3\sum_{\substack{|\bm{\alpha}|=N\\\alpha^{(j)}\ge 1}} \alpha^{(j)}\,p_{\bm{\alpha}}^{(j)} \mathbf{x}^{\bm{\alpha}-\mathbf{e}_j}\equiv 0,
		\end{split}
		\end{equation}
		or equivalently,
		\begin{equation}\label{eq:coDivfree}
		\sum_{j=1}^3 p^{(j)}_{\bm{\beta}+\mathbf{e}_j}(\beta^{(j)}+1)=0,\quad \forall |\bm{\beta}|=N-1.
		\end{equation}
		\item[ii)]
		If the polynomial is curl-free, namely
		\begin{equation}\label{eq:curlfree}
		\nabla\wedge \mathbf{P}_N\equiv 0,
		\end{equation}
		then one has
		\begin{equation}
		\begin{split}
		&\sum_{\substack{|\bm{\alpha}|=N\\\alpha^{(l)}\ge 1}}  \alpha^{(l)}\,p_{\bm{\alpha}}^{(j)} \mathbf{x}^{\bm{\alpha}-\mathbf{e}_l}
		-\sum_{\substack{|\bm{\alpha}|=N\\\alpha^{(j)}\ge 1}}  \alpha^{(j)}\,p_{\bm{\alpha}}^{(l)} \mathbf{x}^{\bm{\alpha}-\mathbf{e}_j}\equiv 0.
		\end{split}
		\end{equation}
		In other words, for any $3$-dimensional index $\bm{\beta}$ such that $|\bm{\beta}|=N-1$, there must hold
		\begin{equation}\label{eq:coCurlfree}
		(\beta^{(l)}+1)p^{(j)}_{\bm{\beta}+\mathbf{e}_l}=(\beta^{(j)}+1)p^{(l)}_{\bm{\beta}+\mathbf{e}_j},\quad j,l=1,2,3.
		\end{equation}
		\item[iii)]
		The polynomial $\mathbf{P}_N$ is harmonic if and only if
		\begin{equation}
		\begin{split}
		&\sum_{l=1}^3\sum_{\substack{|\bm{\alpha}|=N\\\alpha^{(l)}\ge 2}}  \alpha^{(l)}\,(\alpha^{(l)}-1)\,p_{\bm{\alpha}}^{(j)} \mathbf{x}^{\bm{\alpha}-2\mathbf{e}_l}\equiv 0.
		\end{split}
		\end{equation}
		If $N\ge 2$, then for any $3$-dimensional index $\bm{\beta}$ such that $|\bm{\beta}|=N-2$, there must hold
		\begin{equation}\label{eq:idHarmonic}
		\sum_{l=1}^3  (\beta^{(l)}+1)(\beta^{(l)}+2)p^{(j)}_{\bm{\beta}+2\mathbf{e}_l}=0,\quad j=1,2,3.
		\end{equation}
	\end{enumerate}
\end{prop}

In the following, we let the complex scalar function $\mathscr{L}^{(l)}[P_{N}^{(j)}]$ be defined by
\begin{equation}
\mathscr{L}^{(l)}[P_{N}^{(j)}](\bm{\zeta})
:=\int_{\substack{x^{(l)}=0\\\mathbf{x}_{\hat{l}}>0}}
e^{-\mathbf{x}\cdot \bm{\zeta}} P^{(j)}_{N}(\mathbf{x}) \,  d \vartheta_l,\quad j,l=1,2,3,
\end{equation}
where $d\vartheta_l$ signifies the sigma measure of the plane defined by $\{\mathbf{x}\in\mathbb{R}^3; x^{(l)}=0\ \mbox{and}\ \mathbf{x}_{\hat{l}}>0\}$. Then, one has

\begin{lem}\label{lem2.1}
	\begin{equation}\label{eq:eq2.03}
	\begin{split}
	\mathscr{L}^{(l)}[P_{N}^{(j)}](\bm{\zeta})
	=&\sum_{\substack{|\bm{\alpha}|=N\\\alpha^{(l)}= 0}}p_{\bm{\alpha}}^{(j)} \bm{\alpha}! \,\frac{1}{\bm{\zeta}^{\bm{\alpha}_{\hat{l}}+\mathbf{1}_{\hat{l}}}},\quad j,l=1,2,3.
	\end{split}
	\end{equation}
\end{lem}
\begin{proof}
	Notice that
	\begin{equation}
	P_N^{(j)}(\mathbf{x}_{\hat{l}})
	=\sum_{\substack{|\bm{\alpha}|=N\\\alpha^{(l)}= 0}}p_{\bm{\alpha}}^{(j)} \mathbf{x}^{\bm{\alpha}},\quad j,l=1,2,3.
	\end{equation}
	Then
	\begin{equation}
	\begin{split}
	\mathscr{L}^{(l)}[P_{N}^{(j)}](\bm{\zeta})
	=&\int_{\substack{x^{(l)}=0\\\mathbf{x}_{\hat{l}}>0}}
	e^{-\mathbf{x}\cdot \bm{\zeta}} P^{(j)}_{N}(\mathbf{x}) \,  d\vartheta_l\\
	=&\frac{1}{\zeta^{\mathbf{1}_{\hat{l}}}}\int_{\substack{y^{(l)}=0\\\mathbf{y}_{\hat{l}}>0}} e^{-\mathbf{y}_{\hat{l}}\cdot 1}\,  P_{N}^{(j)}(\frac{\mathbf{y}_{\hat{l}}}{\bm{\zeta}})\,d\vartheta_l(\mathbf{y})\\
	=&\frac{1}{\zeta^{\mathbf{1}_{\hat{l}}}}\sum_{\substack{|\bm{\alpha}|=N\\\alpha^{(l)}= 0}}
	\frac{1}{\bm{\zeta}^{\bm{\alpha}_{\hat{l}}}}\,p_{\alpha}^{(j)} \int_{\substack{y^{(l)}=0\\\mathbf{y}_{\hat{l}}>0}} e^{-\mathbf{y}_{\hat{l}}\cdot \mathbf{1}}\,  y^\alpha\,d\vartheta_l(\mathbf{y})\\
	=&\sum_{\substack{|\bm{\alpha}|=N\\\alpha^{(l)}= 0}}p_{\bm{\alpha}}^{(j)} \bm{\alpha}! \,\frac{1}{\bm{\zeta}^{\bm{\alpha}_{\hat{l}}+\mathbf{1}_{\hat{l}}}}.
	\end{split}
	\end{equation}	
\end{proof}

\subsection{Proof of Theorem~\ref{thm:LapTrans}}

Applying the \emph{divergence-free} relation \eqref{eq:divfree}, one has
\begin{equation}
\begin{split}
\mathscr{L}[\bm{\zeta} \cdot \mathbf{P}_{N}](\bm{\zeta})
=&-\int_{\mathcal{K}}  \mathbf{P}_{N}(\mathbf{x})\cdot\nabla e^{-\mathbf{x}\cdot \bm{\zeta}}+e^{-\mathbf{x}\cdot \bm{\zeta}}\nabla\cdot  \mathbf{P}_{N}(\mathbf{x})\,d\mathbf{x}\\
=&\sum_{j=1}^3 \int_{\substack{x^{(j)}=0\\\mathbf{x}_{\hat{j}}>0}}
e^{-\mathbf{x}\cdot \bm{\zeta}} P^{(j)}_{N}(\mathbf{x}) \,  d\vartheta_j\\
=&\sum_{j=1}^3 \mathscr{L}^{(j)}[P_{N}^{(j)}](\bm{\zeta}),
\end{split}
\end{equation}
and hence by Lemma~\ref{lem2.1},
\begin{equation}\label{eq:eq2.3}
\mathscr{L}[\bm{\zeta} \cdot \mathbf{P}_{N}](\bm{\zeta})
=\sum_{j=1}^3 \sum_{\substack{|\bm{\alpha}|=N\\\alpha^{(j)}= 0}}p_{\bm{\alpha}}^{(j)} \bm{\alpha}! \,\frac{1}{\bm{\zeta}^{\bm{\alpha}_{\hat{j}}+\mathbf{1}_{\hat{j}}}}.
\end{equation}
Let
\begin{equation}
\bm{\rho}:=\frac{1}{\bm{\zeta}}.
\end{equation}
Then
\begin{equation}\label{eq:PNhatP}
\mathscr{L}[\bm{\zeta} \cdot \mathbf{P}_{N}](\bm{\zeta})=\mathscr{I}[\mathbf{P}_{N}](\bm{\rho}),
\end{equation}
with
\begin{equation}\label{eq:defPhat}
\mathscr{I}[\mathbf{P}_{N}](\bm{\rho}):=\sum_{j=1}^3 \sum_{\substack{|\bm{\alpha}|=N\\\alpha^{(j)}= 0}}p_{\bm{\alpha}}^{(j)} \bm{\alpha}! \,{\bm{\rho}^{\bm{\alpha}_{\hat{j}}+\mathbf{1}_{\hat{j}}}}
\end{equation}
a homogeneous polynomial of degree $N+2$.

\begin{prop}\label{prop:vanish}
	If $N\ge 1$, and $\mathbf{P}_N$ is such that
	\begin{equation}
	P_N^{(j)}=x^{(j)} P_{N-1}^{(j)},\quad j=1,2,3,
	\end{equation}
	with $P_{N-1}^{(j)}$, $j=1,2,3$, homogeneous polynomials of order $N-1$, that is,
	\begin{equation}\label{eq:vanishOdd}
	p_\alpha^{(j)}=0,\quad \mbox{for all $|\bm{\alpha}|=N$, $\alpha^{(j)}=0$ and $j=1,2,3$},
	\end{equation}
	then one has
	\begin{equation}
	\mathscr{I}[\mathbf{P}_{N}](\bm{\rho})=0,\quad \forall\, \bm{\rho}\in\mathbb{C}^3.
	\end{equation}
\end{prop}
\begin{proof}
This can be readily seen by \eqref{eq:defPhat}.
\end{proof}

The following two Theorems suffice to guarantee Theorem~\ref{thm:LapTrans}.
\begin{thm}\label{thm:divideOdd}
	For a given odd number $N$, $\mathscr{I}[\mathbf{P}_{N}](\bm{\rho})$ can be divided by
	\begin{equation}\label{eq:sss1}
	\sigma(\bm{\rho}):=\sum_{j=1}^{3} \bm{\rho}^{\mathbf{2}_{\hat{j}}},
	\end{equation}
	if and only if \eqref{eq:vanishOdd} holds true. 
\end{thm}

\begin{thm}\label{thm:divideEven}
	For a given even number $N$, $\mathscr{I}[\mathbf{P}_{N}](\bm{\rho})$ can be divided by $\sigma(\bm{\rho})$ in \eqref{eq:sss1} if and only if the polynomial $\mathbf{P}_N$ in \eqref{eq:PNhatP} has the form
	\begin{equation}\label{eq:PformOdd}
	\mathbf{P}_{N}(\mathbf{x})=\left( \mathbf{x}^{\mathbf{1}_{\hat{j}}} P^{(j)}_{N-2}\right) _{j=1}^{3},
	\end{equation}
	where $\mathbf{P}_{N-2}=(P^{(j)}_{N-2})_{j=1}^{3}$ is a homogeneous polynomial of order $N-2$.
\end{thm}

%

\subsection{Proofs of Theorems~\ref{thm:divideOdd} and \ref{thm:divideEven}}

\begin{proof}[Proof of Theorem~\ref{thm:divideOdd}]
	One side of the assertion is just Proposition~\ref{prop:vanish}. We shall show in the rest of the proof that, $\mathscr{I}[\mathbf{P}_{N}](\bm{\rho})$ being dividable by $\sigma(\bm{\rho})$ implies \eqref{eq:vanishOdd}.	
	
	Assume that $\mathscr{I}[\mathbf{P}_{N}](\bm{\rho})$ is dividable by $\sigma(\bm{\rho})$. Then one has
	\begin{equation}\label{eq:divide}
	\mathscr{I}[\mathbf{P}_{N}](\bm{\rho})=\sigma(\bm{\rho}) C(\bm{\rho}),
	\end{equation}
	where
	\begin{equation}\label{eq:eq5.29}
	C(\bm{\rho}):=\sum_{l=0}^{N-2}\left( \rho^{(3)}\right) ^{N-l-2}C_l(\bm{\rho}_{\hat{3}})
	\end{equation}
	is a homogeneous polynomial of degree $N-2$ with
	\begin{equation}
	C_l(\bm{\rho}_{\hat{3}})=\sum_{j=0}^{l}c_{l,j}\left(\rho^{(1)}\right)^{j}\left(\rho^{(2)}\right)^{l-j},\quad 0\le l\le N-2.
	\end{equation}
	
	Write
	\begin{equation}
	\sigma(\bm{\rho})=\left(\rho^{(3)}\right)^2 \sigma_2(\bm{\rho}_{\hat{3}})+\sigma_4(\bm{\rho}_{\hat{3}})
	\end{equation}
	with
	\begin{equation}
	\sigma_2(\bm{\rho}_{\hat{3}})=\left(\rho^{(1)}\right)^2+\left(\rho^{(2)}\right)^2\quad\text{and} \quad \sigma_2(\bm{\rho}_{\hat{3}})=\left(\rho^{(1)}\right)^2\left(\rho^{(2)}\right)^2.
	\end{equation}
	It is straightforwardly derived that
	\begin{equation}\label{eq:sigmaC}
	\begin{split}
	\sigma(\bm{\rho}) C(\bm{\rho})=
	&\left(\rho^{(3)}\right)^{N}\sigma_2 C_0+\sigma_4C^{N-2}
	+\left(\rho^{(3)}\right)^{N-1}\sigma_2 C_1+{\rho^{(3)}}\sigma_4C_{N-3}\\
	&+\sum_{l=2}^{N-2}\left(\rho^{(3)}\right)^{N-l}\left(\sigma_2 C_l +\sigma_4C_{l-2}\right) .
	\end{split}
	\end{equation}
	Write
	\begin{equation}\label{eq:hatP3}
	\begin{split}
	\mathscr{I}[\mathbf{P}_{N}](\bm{\rho})
	=&\sum_{j=1}^2\sum_{\substack{|\bm{\alpha}|=N\\\alpha^{(j)}= 0}}p_{\bm{\alpha}}^{(j)} \bm{\alpha}! \,\left(\rho^{(j_{+})}\right)^{\alpha^{(j_{+})}+1}
	\left(\rho^{(3)}\right)^{\alpha^{(3)}+1}\\
	&+\sum_{\substack{|\bm{\alpha}|=N\\\alpha^{(3)}= 0}}p_{\bm{\alpha}}^{(3)} \bm{\alpha}! \,\left(\rho^{(1)}\right)^{\alpha_1+1}\left(\rho^{(2)}\right)^{\alpha_2+1}\\
	=:& \mathscr{I}[\mathbf{P}_{N}]_{\hat{3}}(\bm{\rho}_{\hat{3}})+\sum_{l=0}^{N}\left(\rho^{(3)}\right)^{N-l+1}\sum_{j=1}^2 \hat{p}^{(j)}_{l+1}\,\left(\rho^{(j)}\right)^{l+1},
	\end{split}	
	\end{equation}
	with the notation
	\begin{equation}
	j_+:=\left\{
	\begin{array}{cc}
	1, &  \text{if }j=2,\\
	2, & \text{if }j=1,
	\end{array}\right.
	\end{equation}
	the homogeneous polynomial
	\begin{equation}\label{eq:eq3.29}
	\mathscr{I}[\mathbf{P}_{N}]_{\hat{3}}(\bm{\rho}_{\hat{3}})=\sum_{\substack{|\bm{\alpha}|=N\\\alpha^{(3)}= 0}}p_{\bm{\alpha}}^{(3)} \bm{\alpha}! \,\left(\rho^{(1)}\right)^{\alpha_1+1}\left(\rho^{(2)}\right)^{\alpha_2+1},
	\end{equation}
	and the coefficients
	\begin{equation}\label{eq:eq3.300}
	\hat{p}^{(j)}_{l+1}=l!\,(N-l)!\,p_{\bm{\alpha}}^{(j_{+})},
	\end{equation}
	where $\alpha^{(j)}=l$, $\alpha^{(3)}=N-l$, for $j=1,2$ and $1\le l\le N$.
	Notice that the order of ${\rho^{(3)}}$ in \eqref{eq:sigmaC} cannot go beyond $N$. Thus it is observed that
	\begin{equation}\label{eq:eq3.31000}
	\hat{p}^{(j)}_{1}=0,\quad j=1,2.
	\end{equation}
	Equating the terms in \eqref{eq:hatP3} and \eqref{eq:sigmaC} with the same power, ranging from $1$ to $N-1$, of ${\rho^{(3)}}$ gives:
	\begin{equation}\label{eq:eq3.3100}
	\sigma_2 C_{1}=\sum_{j=1}^2 \hat{p}^{(j)}_{3}\,\left(\rho^{(j)}\right)^{3},\qquad \sigma_4C_{N-3}= \sum_{j=1}^2 \hat{p}^{(j)}_{N+1}\,\left(\rho^{(j)}\right)^{N+1},
	\end{equation}
	and
	\begin{equation}\label{eq:eq3.3200}
	\sigma_2 C_l +\sigma_4C_{l-2}=\sum_{j=1}^2 \hat{p}^{(j)}_{l+2}\,\left(\rho^{(j)}\right)^{l+2},\qquad 2\le l\le N-2.
	\end{equation}
	The first identity in \eqref{eq:eq3.3100}, along with Lemma~\ref{lem:lem3.2} in the following, readily gives that
	\begin{equation}\label{eq:C1=0}
	C_{1}\equiv 0.
	\end{equation}
	
	\begin{lem}\label{lem:lem3.2}
		Suppose that
		\begin{equation}
		\sigma_2(\bm{\rho}_{\hat{3}})C_l(\bm{\rho}_{\hat{3}})\equiv f_1({\rho^{(1)}})+f_2({\rho^{(2)}}).
		\end{equation}
		If $l$ is an odd number, then
		\begin{equation}
		C_l=0.
		\end{equation}
		Otherwise if $l$ is even, then there exists a constant $a_l$ such that
		\begin{equation}
		c_{l,2j}=(-1)^j a_l,\quad c_{l,2j+1}=0,\quad 0\le j\le l/2,
		\end{equation}
		and
		\begin{equation}
		\sigma_2(\bm{\rho}_{\hat{3}})C_l(\bm{\rho}_{\hat{3}})=a_l\left( \left(\rho^{(2)}\right)^{l+2}+ (-1)^{l/2} \left(\rho^{(1)}\right)^{l+2}\right). 
		\end{equation}
	\end{lem}
	
	The proof of Lemma~\ref{lem:lem3.2} is postponed to the end of the current section, and we proceed with the present proof of Theorem~\ref{thm:divideOdd}. Applying an induction argument on \eqref{eq:eq3.3200} with the help of \eqref{eq:C1=0} and Lemma~\ref{lem:lem3.2}, one can derive that
	\begin{equation}
	C_l\equiv 0, \quad \mbox{for every odd integer $l\in[1,N-2]$},
	\end{equation}
	and hence by combining \eqref{eq:eq3.31000}-\eqref{eq:eq3.3200} that
	\begin{equation}
	\hat{p}^{(j)}_{l}=0,\, j=1,2, \quad \mbox{for every odd $l\in[1,N]$}.
	\end{equation}
	Therefore, it is obtained by recalling \eqref{eq:eq3.300} that
	\begin{equation}\label{eq:eq3.410}
	p_{(0,l,N-l)}^{(1)}=p_{(l,0,N-l)}^{(2)}=0, \quad \mbox{for every even $l\in[0,N-1]$}.
	\end{equation}
	Furthermore, it is noted that \eqref{eq:eq5.29} can be also reformulated as
		\begin{equation}
		C(\bm{\rho})=\sum_{l=0}^{N-2}\left( \rho^{(1)}\right) ^{N-l-2}\tilde{C}_l(\bm{\rho}_{\hat{1}}).
		\end{equation}
		Then by exchanging the roles of ${\rho^{(3)}}$ and ${\rho^{(1)}}$ and repeating all the above arguments starting from \eqref{eq:eq5.29}, one can obtain analogously that
	\begin{equation}\label{eq:eq3.42}
	p_{(N-l,0,l)}^{(2)}=p_{(N-l,l,0)}^{(3)}=0, \quad \mbox{for every even $l\in[0,N-1]$}.
	\end{equation}
	Similarly by exchanging the roles of ${\rho^{(3)}}$ and ${\rho^{(2)}}$, one also has
	\begin{equation}\label{eq:eq3.43}
	p_{(l,N-l,0)}^{(3)}=p_{(0,N-l,l)}^{(1)}=0, \quad \mbox{for every even $l\in[0,N-1]$}.
	\end{equation}
	
	Here, we emphasize that the proof so far does not depend on the odevity of the number $N$, and this fact shall be useful in the proof of Lemma~\ref{lem:lem3.300} in the following. However, if $N$ is odd, then the condition \eqref{eq:vanishOdd} is clearly implied by \eqref{eq:eq3.410}-\eqref{eq:eq3.43}, which completes the proof. 
	
\end{proof}

\begin{proof}[Proof of Theorem~\ref{thm:divideEven}]
	The assertion is equivalent to that, if $\mathscr{I}[\mathbf{P}_{N}](\bm{\rho})$ can be divided by $\sigma(\bm{\rho})$, then
	\begin{equation}\label{eq:eq3.50}
	p^{(j)}_{\alpha}=0\quad\mbox{with $\alpha^{(l)}=0$, $l\neq j$},
	\end{equation}
	for all $j=1,2,3$.
	
	We first verify \eqref{eq:eq3.50} for $\bm{\alpha}\in(2\mathbb{N})^3$, namely,
	\begin{equation}\label{eq:eq3.510}
	p^{(2)}_{(0,l,N-l)}=p^{(3)}_{(0,l,N-l)}=p^{(3)}_{(N-l,0,l)}=p^{(1)}_{(N-l,0,l)}=p^{(1)}_{(l,N-l,0)}=p^{(2)}_{(l,N-l,0)}=0,
	\end{equation}
	for all even integers $l\in[0,N]$.
	This is guaranteed by the following stronger result.
	\begin{lem}\label{lem:lem3.300}
		One has for the given even integer $N$ that
		\begin{equation}\label{eq:eq2add}
		p^{(j)}_{(2l,2n,N-2n-2l)}=0,\quad	0\le l+n\le N/2,\, j=1,2,3.
		\end{equation}
	\end{lem}
		
	We postpone the proof of Lemma~\ref{lem:lem3.300} and proceed with the proof of Theorem~\ref{thm:divideEven}. It remains to verify \eqref{eq:eq3.510} for odd numbers $l$. To ease the exposition, we only verify the case
	\begin{equation}\label{eq:eq3odd}
	p^{(3)}_{(0,2l+1,N-2l-1)}=0,\quad	0\le l\le (N-1)/2.
	\end{equation}
	Using \eqref{eq:coCurlfree} with $\beta=(0,2l,N-2l-1)$ for each $l$, it can be shown that \eqref{eq:eq3odd} holds true by virtue of the fact that
	\begin{equation}\label{eq:eq2even}
	p^{(2)}_{(0,2l,N-2l)}=0,\quad	0\le l\le (N-1)/2,
	\end{equation}
	which was established in Lemma~\ref{lem:lem3.300}.
	
	The proof is complete.
\end{proof}

\begin{proof}[Proof of Lemma~\ref{lem:lem3.300}]
Without loss of generality, we only prove the case when $j=2$, and the other cases can be proved by following a similar argument. We shall verify  
		\begin{equation}\label{eq:eq2add1}
		p^{(2)}_{(2l,2n,N-2n-2l)}=0,\quad	0\le l+n\le N/2,
		\end{equation}
		by induction with respect to $n$.
		
		First, when $n=0$, \eqref{eq:eq2add1} can be readily seen from \eqref{eq:eq3.410} and \eqref{eq:eq3.42} in the proof of Theorem~\ref{thm:divideOdd}. Assume that \eqref{eq:eq2add1} holds true with a fixed integer $n\in[0,N/2)$ and all $0\le l\le N/2-n$. We next show that
		\begin{equation}\label{eq:eq3.580}
		p^{(2)}_{(2l,2n+2,N-2n-2l-2)}=0,\quad \text{for all }	0\le l\le N/2-n-1.
		\end{equation}
		Notice by the assumption for $n$ that
		\begin{equation}
		p^{(2)}_{(2l,2n,N-2n-2l)}=0,\quad 0\le l\le N/2-n,
		\end{equation}
		and that
		\begin{equation}
		p^{(2)}_{(2l+2,2n,N-2n-2l-2)}=0,\quad 0\le l\le N/2-n-1.
		\end{equation}
		Then \eqref{eq:eq3.580} can be obtained by an application of \eqref{eq:idHarmonic} with $\beta=(2l,2n,N-2n-2l-2)$.
		
		The proof is complete. 
		
	\end{proof}

\begin{proof}[Proof of Lemma~\ref{lem:lem3.2}]
	It is directly calculated that
	\begin{equation}
	\begin{split}
	\sigma_2(\bm{\rho}_{\hat{3}})C_l(\bm{\rho}_{\hat{3}})=
	&c_{l,0}\left(\rho^{(2)}\right)^{l+2}+c_{l,l}\left(\rho^{(1)}\right)^{l+2}\\
	&+c_{l,1}{\rho^{(1)}} \left(\rho^{(2)}\right)^{l+1}+c_{l,l-1}\left(\rho^{(1)}\right)^{l+1} {\rho^{(2)}}\\
	&+\sum_{j=2}^{l} \left( c_{l,j}+c_{l,j-2}\right) \left(\rho^{(1)}\right)^{j} \left(\rho^{(2)}\right)^{l-j+2}.
	\end{split}
	\end{equation}
	Therefore,
	\begin{equation}\label{eq:eq3.310}
	c_{l,1}=c_{l,l-1}=0,
	\end{equation}
	and
	\begin{equation}\label{eq:eq3.320}
	c_{l,j}+c_{l,j-2}=0,\quad 2\le j\le l.
	\end{equation}
	
	If $l$ is an odd number, it is observed from \eqref{eq:eq3.310} and \eqref{eq:eq3.320} by recursion that
	\begin{equation}
	c_{l,j}=0, \quad 0\le j\le l,
	\end{equation}
	and hence $C_l$ is identically zero.
	Otherwise if $l$ is even, then \eqref{eq:eq3.320} together with \eqref{eq:eq3.310} implies $c_{l,j}=0$ for all odd numbers $j$. While for even numbers, one has
	\begin{equation}
	c_{l,2j}=(-1)^j c_{l,0},\quad  0\le j\le l/2,
	\end{equation}
	and
	\begin{equation} \sigma_2(\bm{\rho}_{\hat{3}})C_l(\bm{\rho}_{\hat{3}})=c_{l,0}\left(\rho^{(2)}\right)^{l+2}+c_{l,l}\left(\rho^{(1)}\right)^{l+2}.
	\end{equation}

The proof is complete. 
\end{proof}

\section{Proof of Theorem~\ref{thm:admiEquiv}}

\subsection{Preliminaries}
\begin{lem}[\cite{CoK12}]\label{lem:lem2.30}
	For any integers $l\ge 0$ and $|m|\le l$, $Y_l^{m}(\hat{\mathbf{x}})|\mathbf{x}|^l$ is a homogeneous polynomial of order $l$.
\end{lem}

\begin{cor}\label{cor:cor2.2}
	For each pair of integers $(l,m)$ with $l\ge 0$ and $|m|\le l+1$, $\mathbf{I}_{l}^{m}(\hat{\mathbf{x}})|\mathbf{x}|^l$ is a homogeneous polynomial of order $l$, $\mathbf{T}_{l+1}^{m}(\hat{\mathbf{x}})|\mathbf{x}|^{l+1}$ is a homogeneous polynomial of order $l+1$,
	and $\mathbf{N}_{l+2}^{m}(\hat{\mathbf{x}})|\mathbf{x}|^{l+2}$ is a homogeneous polynomial of order $l+2$.
\end{cor}
\begin{proof}
	The property for $\mathbf{I}_{l}^{m}(\hat{\mathbf{x}})|\mathbf{x}|^l$ can be readily seen by noting the identity
	\begin{equation}\label{eq:eq2.160}
	\sqrt{(l+1)(2l+3)}\,\mathbf{I}_{l}^{m}(\hat{\mathbf{x}})|\mathbf{x}|^l
	= \nabla \left( Y_{l+1}^{m}(\hat{\mathbf{x}})\,|\mathbf{x}|^{l+1}\right),
	\end{equation}
	along with the statement in Lemma~\ref{lem:lem2.30}.
	Furthermore, it is noticed from the definitions \eqref{eq:Tlm} and \eqref{eq:Ilm} that
	\begin{equation}
	\sqrt{l(l+1)}\,\mathbf{T}_{l}^{m}(\hat{\mathbf{x}})
	=\sqrt{l(2l+1)}\mathbf{I}_{l-1}^{m}\wedge \hat{\mathbf{x}},
	\end{equation}
	and hence by \eqref{eq:eq2.160} that
	\begin{equation}
	\sqrt{l+2}\,\mathbf{T}_{l+1}^{m}(\hat{\mathbf{x}})|\mathbf{x}|^{l+1}
	=\sqrt{2l+3}\,\mathbf{I}_{l}^{m}(\hat{\mathbf{x}})|\mathbf{x}|^{l} \wedge \mathbf{x},
	\end{equation}
	where the RHS is clearly a homogeneous polynomial since we have verified in \eqref{eq:eq2.160} that $\mathbf{I}_{l}^{m}(\hat{\mathbf{x}})|\mathbf{x}|^l$ is a homogeneous polynomial of order $l$.
	Similarly, it is obtained that
	\begin{equation}
	\begin{split}
	&\mathbf{N}_{l+2}^{m}(\hat{\mathbf{x}})|\mathbf{x}|^{l+2}\\
	=&\frac{|\mathbf{x}|^{l+2}}{\sqrt{(l+2)(2l+3)}}
	\left(- \nabla_{\mathrm{S}}Y_{l+1}^{m}(\hat{\mathbf{x}})+lY_{l+1}^{m}(\hat{\mathbf{x}})\hat{\mathbf{x}}\right)\\
	=&-\sqrt{\frac{l+1}{l+2}}\,|\mathbf{x}|^{2}\left( \mathbf{I}_{l}^{m}(\hat{\mathbf{x}})|\mathbf{x}|^{l}\right)
	+\frac{2l+1}{\sqrt{(l+2)(2l+3)}}\left( |\mathbf{x}|^{l+1}Y_{l+1}^{m} (\hat{\mathbf{x}})\right)\mathbf{x},
	\end{split}
	\end{equation}
	which clearly shows that $\mathbf{N}_{l+2}^{m}(\hat{\mathbf{x}})|\mathbf{x}|^{l+2}$ is a homogeneous polynomial of order $l+2$.
\end{proof}

\begin{lem}\label{lem:a2}
	One has
	\begin{equation}
	N_{\mathbf{H}_{l,m}}=l,\,N_{\mathbf{E}_{l,m}}=l+1,\quad \forall m,l.
	\end{equation}
	Moreover,
	\begin{equation}\label{eq:HlmP}
	\left(\mathbf{P}_l[\mathbf{H}_{l,m}]\right) (\mathbf{x})=-\mathtt{i}\sqrt{\varepsilon_0/\mu_0}\sqrt{\frac{l+2}{2l+3}}\,\frac{2^l\,l!}{(2l+1)!} \, \mathbf{I}_{l}^{m}(\hat{\mathbf{x}})|\mathbf{x}|^l,
	\end{equation}
	and
	\begin{equation}\label{eq:ElmP}
	\begin{split}
	\left( \mathbf{P}_{l+1}[\mathbf{E}_{l,m}]\right) (\mathbf{x})
	&=\frac{1}{2l+3}\frac{2^{l}\,l!}{(2l+1)!} \, \mathbf{T}_{l+1}^{m}(\hat{\mathbf{x}})|\mathbf{x}|^{l+1}\\
	&=\frac{1}{\sqrt{(l+2)(2l+3)}}\frac{2^{l}\,l!}{(2l+1)!}\mathbf{I}_{l}^{m}(\hat{\mathbf{x}})|\mathbf{x}|^{l} \wedge \mathbf{x}.
	\end{split}	
	\end{equation}
\end{lem}
\begin{proof}
	Recall that the spherical Bessel function $j_l(t)$ is analytic for all $t\in\R$, and hence the formula \eqref{eq:jl} is actually the Taylor series of $j_l(t)$.
	It is then observed by the definition \eqref{eq:Elm} of $\mathbf{E}_{l,m}$, along with the form \eqref{eq:jl} of $j_l(t)$, that
	\begin{equation}\label{eq:ElmTaylor}
	\mathbf{E}_{l,m}=\sum_{n=0}^{\infty}\frac{ (l+n+1)!\, 2^{l+1} }{ n!\,(2l+2n+3)! }\,k^{l+2n+1}|\mathbf{x}|^{2n}\left( \mathbf{T}_{l+1}^{m}(\hat{\mathbf{x}})|\mathbf{x}|^{l+1}\right),
	\end{equation}
	where, deduced by Corollary~\ref{cor:cor2.2}, each term in the summation is a homogeneous polynomial of order $(2n+l+1)$.
	Therefore, the formula \eqref{eq:ElmTaylor} is actually the Taylor series of $\mathbf{E}_{l,m}$. Now, \eqref{eq:ElmP} is a direct consequence of the above fact. The property \eqref{eq:HlmP} can be similarly verified.

The proof is complete. 
\end{proof}

\subsection{Proof of Theorem~\ref{thm:admiEquiv}}
Let $(\mathbf{E}^0,\mathbf{H}^0)$ be a pair of functions satisfying \eqref{eq:EHin_eqn} in a neighborhood of $\mathbf{x}_0=\mathbf{0}$.
Lemma~\ref{lem:lem2.300} then indicates that $(\mathbf{E}^0,\mathbf{H}^0)$ can be expressed as
\begin{equation}
(\mathbf{E}^0,\mathbf{H}^0)=\sum_{l=l_0}^{\infty}\sum_{|m|\le l+1}
\left( a_{l,m}(\mathbf{E}\mathbf{H})_{l,m} + b_{l,m}(\mathbf{H}\mathbf{E})_{l,m}\right)
\end{equation}
with $l_0\in\N$ and the coefficients $a_{l,m}\in\C$ such that
\begin{equation}
\sum_{|m|\le l_0+1}
\left( a^2_{l_0,m}+b^2_{l_0,m}\right) \neq 0.
\end{equation}	
Assume without loss of generality that
\begin{equation}
\sum_{|m|\le l_0+1} a^2_{l_0,m} \neq 0.
\end{equation}
Then it is deduced by Lemma~\ref{lem:a2} that $N_{\mathbf{H}^0}=l_0$ and
\begin{equation}
\left(\mathbf{P}_{l_0}[\mathbf{H}^0]\right) (\mathbf{x})=-\mathtt{i}\sqrt{\varepsilon_0/\mu_0}\sqrt{\frac{l_0+2}{2l_0+3}}\,\frac{2^l_0\,l_0!}{(2l_0+1)!} \, \sum_{|m|\le l_0+1}a_{l_0,m}\mathbf{I}_{l_0}^{m}(\hat{\mathbf{x}})|\mathbf{x}|^{l_0}.
\end{equation}
Theorem~\ref{thm:admiEquiv} is then a straightforward consequence of the following lemma.


\begin{lem}\label{lem:a1}
	For fixed integers $l\ge 0$ and $m$, with $(l+1)\, \mathrm{mod}\, 2\le m\le [(l+1)/2]$, the vector field defined as
	\begin{equation}\label{eq:eq2.29}
	\mathbf{P}(\mathbf{x}):= |\mathbf{x}|^{l} \left(  \mathbf{I}_{l}^{2m}(\hat{\mathbf{x}})+(-1)^{l} \mathbf{I}_{l}^{-2m}(\hat{\mathbf{x}})\right) ,
	\end{equation}
	is a homogeneous polynomial of order $l$, which satisfies the conditions (I) and (II) introduced in Definition~\ref{defn:admiType}.
	
	Conversely, up to a linear combination with respect to $m$, $P$ defined by \eqref{eq:eq2.29} is the only case for $\mathbf{P}_{N_{\mathbf{H}_{\mathbf{E}}}}[\mathbf{H}_{\mathbf{E}}]$ satisfying (I) and (II) in Definition~\ref{defn:admiType}.
\end{lem}
The proof of Lemma~\ref{lem:a1} is a bit lengthy with tedious calculations, which we shall accomplish in the whole next subsection.

\subsection{Proof of Lemma \ref{lem:a1}}
Noticing for any $l$ and $m$ that $\mathbf{I}_{l}^{-m}$ is just the conjugate of $\mathbf{I}_{l}^{m}$, we shall always assume in the sequel that $m\ge 0$.
Denote
\begin{equation}
\tilde{c}_{l,m}:= \sqrt{\frac{2l+1}{4\pi}\frac{(l-m)!}{(l+m)!}}\quad \text{and}\quad
c_{l,m}:=\frac{\tilde{c}_{l+1,m}}{\sqrt{(l+1)(2l+3)}} .
\end{equation}
Then
\begin{equation}
Y_{l}^{m}(\theta,\varphi) =\tilde{c}_{l,m} P_l^{m}(\cos{\theta}) e^{\mathtt{i}m\varphi},
\end{equation}
and hence
\begin{equation}
\nabla_{\mathrm{S}} Y_{l}^{m}(\theta,\varphi)
= \tilde{c}_{l,m}\left(  \frac{dP_l^{m}(\cos{\theta})}{d\theta}\vec{\theta}+\mathtt{i}\frac{m}{\sin \theta}P_l^{m}(\cos{\theta})\vec{\varphi}\right) e^{\mathtt{i}m\varphi},
\end{equation}	
with
\begin{eqnarray}
\vec{r}&=&\sin\theta \cos\varphi\, \mathbf{e}_1+ \sin\theta \sin\varphi \,\mathbf{e}_2+\cos\theta \,\mathbf{e}_3,\\
\vec{\theta}&=&\cos\theta \cos\varphi \,\mathbf{e}_1+ \cos\theta \sin\varphi\, \mathbf{e}_2-\sin\theta \,\mathbf{e}_3,\\
\vec{\varphi}&=&-\sin\varphi\, \mathbf{e}_1+ \cos\varphi \,\mathbf{e}_2	.
\end{eqnarray}	
Thus by the definition \eqref{eq:Ilm} of $\mathbf{I}_{l}^{m}$ one has
\begin{equation}\label{eq:Ilm_r}
\begin{split}
\frac{1}{c_{l,m}}\mathbf{I}_{l}^{m}(\hat{\mathbf{x}})|\mathbf{x}|^l
=&(l+1)r^{l}P_{l+1}^{m}(\cos{\theta})e^{\mathtt{i}m\varphi}\vec{r}\\
&+r^{l} \frac{dP_{l+1}^{m}(\cos{\theta})}{d\theta}e^{\mathtt{i}m\varphi}\vec{\theta}
+\mathtt{i}r^{l}\frac{m}{\sin \theta}P_{l+1}^{m}(\cos{\theta})e^{\mathtt{i}m\varphi}\vec{\varphi}.
\end{split}
\end{equation}
Notice that the associated Legendre polynomial $P_{l+1}^{m}$ is defined by (cf.\cite{CoK12})
\begin{equation}
P_{l+1}^{m}(t)=(1-t^2)^{m/2}\,\frac{d^{m}P_{l+1}(t)}{dt^{m}},
\end{equation}
where $P_l$ is the Legendre polynomial of degree $l$, which can be expressed as
\begin{equation}\label{eq:Pl}
P_{l}(t)=\sum_{n=0}^{l}p_l^{(l-n)}t^{n}.
\end{equation}

\begin{lem}
	One has for any integers $l$ and $n$ with $0\le n\le l$ that
	\begin{equation}
	p_{l}^{(n)}=0,\quad \mbox{when $n$ is odd}.
	\end{equation}
\end{lem}
\begin{proof}
	This is a direct consequence of the fact that the Legendre polynomial $P_l$ is an even function if $l$ even, and an odd function if $l$ is odd (cf.\cite{CoK12}).
\end{proof}

It is obtained straightforwardly from \eqref{eq:Pl} that
\begin{equation}
\frac{d^{m}P_{l+1}(t)}{dt^{m}}=\sum_{n=m}^{l+1}p_{l+1}^{(l-n+1)}\frac{n!}{(n-m)!}t^{n-m},
\end{equation}
and hence that
\begin{equation}\label{eq:Plm}
\begin{split}
P_{l+1}^{m}(\cos\theta)
=&\sum_{n=m}^{l+1}\frac{p_{l+1}^{(l-n+1)}n!}{(n-m)!}(\cos\theta)^{n-m}(\sin\theta)^{m}	\\
=&\sum_{n=0}^{l-m+1}\frac{p_{l+1}^{(n)}(l-n+1)!}{(l-m-n+1)!}(\cos\theta)^{l-m-n+1}(\sin\theta)^{m}.\\
=&\sum_{n=0}^{l'}\frac{p_{l+1}^{(2n)}(l-2n+1)!}{(l-m-2n+1)!}(\cos\theta)^{l-m-2n+1}(\sin\theta)^{m},
\end{split}	
\end{equation}
with the non-negative integer
\begin{equation}\label{eq:l'}
l':=[(l-m+1)/2].
\end{equation}
Next, depending on the odevity of the integer $(l-m+1)$, we divide the rest of the arguments into two parts.
	
\subsubsection{The number $(l-m+1)$ is odd}	
In the case $(l-m+1)$ is odd, one has $l'=(l-m)/2$, and $0\le m\le l$.
Moreover, it is observed from \eqref{eq:Plm} that
\begin{equation}\label{eq:PlmOdd}
\begin{split}
P_{l+1}^{m}(\cos\theta)
=&\sum_{n=0}^{l'}p_{l+1}^{(2l'-2n)}\frac{(m+2n+1)!}{(2n+1)!}(\cos\theta)^{2n+1}(\sin\theta)^{m} \\
=&\sum_{n=0}^{l'}(-1)^nA_{n,m,l'}(\sin\theta)^{m+2n}\cos\theta,
\end{split}	
\end{equation}
where the coefficients $A_{n,m,l'}$ are given by
\begin{equation}
A_{n,m,l'}:=\sum_{t=n}^{l'}p_{l+1}^{(2l'-2t)}\frac{(m+2t+1)!}{(2t+1)!}
\complement_{t}^{n},
\end{equation}	
with
\begin{equation}
\complement_{t}^{n}=\frac{t!}{n!(n-t)!}.
\end{equation}
One also has by \eqref{eq:PlmOdd} that
\begin{equation}\label{eq:dPlmOdd}
\begin{split}
\frac{dP_{l+1}^{m}(\cos{\theta})}{d\theta}
=&\sum_{n=0}^{l'}(-1)^n(m+2n)A_{n,m,l'}(\sin\theta)^{m+2n-1}\\
&-\sum_{n=0}^{l'}(-1)^n(m+2n+1)A_{n,m,l'}(\sin\theta)^{m+2n+1}.
\end{split}
\end{equation}
Inserting the equations \eqref{eq:PlmOdd} and \eqref{eq:dPlmOdd} into \eqref{eq:Ilm_r} one can arrive at
\begin{equation}\label{eq:eq2.63}
\begin{split}
&\frac{1}{c_{l,m}}\left( \mathbf{I}_{l}^{m}(\hat{\mathbf{x}})\right) ^{(1)}|\mathbf{x}|^l\\
=&2\sum_{n=1}^{l'}(-1)^n \tilde{A}_{n,m,l'}(\sin\theta)^{m+2n-1}\cos\theta \cos\varphi e^{\mathtt{i}m\varphi}r^{l}\\
&+m\sum_{n=0}^{l'}(-1)^nA_{n,m,l'}(\sin\theta)^{m+2n-1}\cos\theta e^{\mathtt{i}(m-1)\varphi}r^{l}\\
=:&2\left( \mathbf{K}_{l}^{m}(\mathbf{x})\right) ^{(1)}
+m\left( \mathbf{K}_{l}^{m}(\mathbf{x})\right) ^{(2)},
\end{split}
\end{equation}
with the coefficients
\begin{equation}
\tilde{A}_{n,m,l'}:=nA_{n,m,l'}-(l'-n+1)A_{n-1,m,l'}.
\end{equation}
Analogously, one has
\begin{equation}\label{eq:eq2.67}
\begin{split}
&\frac{1}{c_{l,m}}\left( \mathbf{I}_{l}^{m}(\hat{\mathbf{x}})\right) ^{(2)}|\mathbf{x}|^l\\
=&2\sum_{n=1}^{l'}(-1)^n\tilde{A}_{n,m,l'}(\sin\theta)^{m+2n-1}\cos\theta \sin\varphi e^{\mathtt{i}m\varphi}r^{l}\\
&+\mathtt{i}m\sum_{n=0}^{l'}(-1)^nA_{n,m,l'}(\sin\theta)^{m+2n-1}\cos\theta e^{\mathtt{i}(m-1)\varphi}r^{l}\\
=&2\left( \mathbf{K}_{l}^{m}(\mathbf{x})\right) ^{(1)}
+\mathtt{i}m\left( \mathbf{K}_{l}^{m}(\mathbf{x})\right) ^{(2)},
\end{split}
\end{equation}
and
\begin{equation}\label{eq:eq2.66}
\begin{split}
\frac{1}{c_{l,m}}\left( \mathbf{I}_{l}^{m}(\hat{\mathbf{x}})\right) ^{(3)}|\mathbf{x}|^l
=&\sum_{n=0}^{l'}(-1)^nA'_{n,m,l'}(\sin\theta)^{m+2n} e^{\mathtt{i}m\varphi}r^{l}\\
=:&\left( \mathbf{K}_{l}^{m}(\mathbf{x})\right) ^{(3)},
\end{split}
\end{equation}
with the coefficients
\begin{equation}
A'_{0,m,l'}:=(2l'+1)A_{0,m,l'},
\end{equation}
and
\begin{equation}
A'_{n,m,l'}:=(2l'-2n+1)A_{n,m,l'}+2(l'-n+1)A_{n-1,m,l'},\quad 1\le n\le l'.
\end{equation}	
	
By recalling the change of coordinates:
\begin{equation}\label{eq:Coordinate}
\begin{split}
x^{(1)}&=r \sin\theta \cos\varphi,\\
x^{(2)}&=r \sin\theta \sin\varphi,\\
x^{(3)}&=r \cos\theta,
\end{split}
\end{equation}
one can deduce further from \eqref{eq:eq2.63} and \eqref{eq:eq2.66} that
\begin{equation}\label{eq:Klm1x}
\left( \mathbf{K}_{l}^{m}(\mathbf{x})\right) ^{(1)}
=\sum_{n=1}^{l'}(-1)^n \tilde{A}_{n,m,l'} x^{(1)}x^{(3)}  \left(x^{(1)}+\mathtt{i}x^{(2)} \right)^{m}r_3^{2n-2}r^{2l'-2n},
\end{equation}
\begin{equation}
\left( \mathbf{K}_{l}^{m}(\mathbf{x})\right) ^{(2)}
=\sum_{n=0}^{l'}(-1)^nA_{n,m,l'}x^{(3)}  \left(x^{(1)}+\mathtt{i}x^{(2)} \right)^{m-1}r_3^{2n}r^{2l'-2n},
\end{equation}	
and
\begin{equation}\label{eq:Klm3x}
\left( \mathbf{K}_{l}^{m}(\mathbf{x})\right) ^{(3)}
=\sum_{n=0}^{l'}(-1)^nA'_{n,m,l'} \left(x^{(1)}+\mathtt{i}x^{(2)} \right)^{m}r_3^{2n}r^{2l'-2n},
\end{equation}
with the notation
\begin{equation}
r_3:=\sqrt{\left( x^{(1)}\right) ^2+\left( x^{(2)}\right) ^2}.
\end{equation}	

To proceed further with the proof of Lemma~\ref{lem:a1}, we study the factors
\begin{equation}\label{eq:factor}
\mathbf{x}^{\bm{\alpha}}\quad \mbox{with $\bm{\alpha}\in\{0,1\}^3$},
\end{equation}
in the Cartesian components of $\mathbf{I}_{l}^{m}(\hat{\mathbf{x}})|\mathbf{x}|^l$.	
Let, for the time being, $2\le m\le l$.
The following table, which is deducted from \eqref{eq:Klm1x}--\eqref{eq:Klm3x}, lists all the factors of the form \eqref{eq:factor} for $\left(x^{(1)}+\mathtt{i}x^{(2)} \right)^{m}$ and $\mathbf{K}_{l}^{m}(\mathbf{x})$:
\begin{equation*}
\begin{array}{c}
\qquad\begin{array}{ccc} 	
&  \qquad \mbox{$m$ is even}
&  \qquad \mbox{$m$ is odd}
\\ \hline
& \qquad\begin{array}{ll}
\Re & \qquad \Im
\end{array}
& \qquad\begin{array}{rr}
\Re &  \qquad\Im
\end{array}
\end{array}
\\
\begin{array}{c|l|l|l|l}\hline
\left(x^{(1)}+\mathtt{i}x^{(2)} \right)^{m}
& \text{None} & x^{(1)}x^{(2)} & x^{(1)} & x^{(2)}
\\
\left( \mathbf{K}_{l}^{m}(\mathbf{x})\right) ^{(1)}
& x^{(1)}x^{(3)} & x^{(1)}x^{(2)}x^{(3)}
& x^{(1)}x^{(3)}
& x^{(1)}x^{(2)}x^{(3)}
\\
\left( \mathbf{K}_{l}^{m}(\mathbf{x})\right) ^{(2)}
& x^{(1)}x^{(3)} & x^{(2)}x^{(3)}
& x^{(3)}	& x^{(1)}x^{(2)}x^{(3)}
\\
\left( \mathbf{K}_{l}^{m}(\mathbf{x})\right) ^{(3)}
& \text{None} & x^{(1)}x^{(2)} & x^{(1)} & x^{(2)}
\end{array}
\end{array}
\end{equation*}
Thus, using \eqref{eq:eq2.63}, \eqref{eq:eq2.67} and \eqref{eq:eq2.66} one can obtain further for the factors of the form \eqref{eq:factor} for the homogeneous polynomial $\mathbf{I}_{l}^{m}(\hat{\mathbf{x}})|x|^l$ in the following table,
\begin{equation*}
\begin{array}{c}
\qquad\begin{array}{ccc} 	
&  \qquad \mbox{$m$ is even}
&  \qquad \mbox{$m$ is odd}
\\ \hline
& \qquad\begin{array}{ll}
\Re & \qquad \Im
\end{array}
& \qquad\begin{array}{rr}
\Re &  \qquad\Im
\end{array}
\end{array}
\\
\begin{array}{c|l|l|l|l}\hline
\left( \mathbf{I}_{l}^{m}(\hat{\mathbf{x}})\right) ^{(1)}|x|^l
& x^{(1)}x^{(3)} & x^{(2)}x^{(3)}
& x^{(3)}& x^{(1)}x^{(2)}x^{(3)}
\\
\left( \mathbf{I}_{l}^{m}(\hat{\mathbf{x}})\right) ^{(2)}|x|^l
& x^{(3)} & x^{(1)}x^{(3)}
& x^{(1)}x^{(3)}	& x^{(3)}
\\
\left( \mathbf{I}_{l}^{m}(\hat{\mathbf{x}})\right) ^{(3)}|x|^l
& \text{None} & x^{(1)}x^{(2)} & x^{(1)} & x^{(2)}
\end{array}
\end{array}
\end{equation*}
There are two special cases when $m=0$ or $m=1$ that one needs to take care of:	
\begin{equation*}
\begin{array}{c}
\qquad\begin{array}{ccc} 	
&  \qquad m=0
&  \qquad m=1
\\ \hline
&
\qquad\Re
& \begin{array}{rr}
\qquad\Re &  \qquad\Im
\end{array}
\end{array}
\\
\begin{array}{c|l|l|l}\hline
\left( \mathbf{I}_{l}^{m}(\hat{\mathbf{x}})\right) ^{(1)}|x|^l
& x^{(1)}x^{(3)} & x^{(3)}& x^{(1)}x^{(2)}x^{(3)}
\\
\left( \mathbf{I}_{l}^{m}(\hat{\mathbf{x}})\right) ^{(2)}|x|^l
& x^{(1)}x^{(3)} & x^{(1)}x^{(3)}	& x^{(3)}
\\
\left( \mathbf{I}_{l}^{m}(\hat{\mathbf{x}})\right) ^{(3)}|x|^l
& \text{None}  & x^{(1)} & x^{(2)}
\end{array}
\end{array}
\end{equation*}	
Therefore one can deduce from the last two tables that, for $(l+m+1)$ odd, the case \eqref{eq:Podd} cannot occur, and that \eqref{eq:Peven} occurs if and only if $m$ is positive even (and hence $l$ is also even), and the corresponding term is
\begin{equation}\label{eq:Odd}
\Im \mathbf{I}_{l}^{m}(\mathbf{x})=- \frac{\mathtt{i}}{2}\left( \mathbf{I}_{l}^{m}(\mathbf{x})-\mathbf{I}_{l}^{-m}(\mathbf{x})\right),
\quad \mbox{$l,m$ are even, $2\le m\le l$}.
\end{equation}

\subsubsection{The integer $(l-m+1)$ is even}
In this case, we have from \eqref{eq:l'} that $l'=(l-m+1)/2$.
Next we distinguish between the cases when $m=l+1$ and when $0\le m < l$.
We first deal with the former case when $m=l+1$, i.e., $l'=0$.
It is directly observed from \eqref{eq:Plm} that
\begin{equation}
P_{l+1}^{l+1}(\cos\theta)
=A_l(\sin\theta)^{l+1},
\end{equation}
with the coefficient $A_l:=(l+1)!p_{l+1}^{(0)}$,
and hence that
\begin{equation}
\begin{split}
\frac{dP_{l+1}^{l+1}(\cos{\theta})}{d\theta}
=&(l+1)A_l(\sin\theta)^{l}\cos\theta.
\end{split}	
\end{equation}
Thus combing \eqref{eq:Ilm_r} one further has
\begin{equation}
\begin{split}
\frac{1}{c_{l,l+1}}\left( \mathbf{I}_{l}^{l+1}(\hat{\mathbf{x}})\right) ^{(1)}|\mathbf{x}|^l
=&(l+1)A_l(\sin\theta)^{l} e^{\mathtt{i}l\varphi}r^{l}\\
=&(l+1)A_l \left( x^{(1)}+\mathtt{i}x^{(2)}\right) ^{l},
\end{split}
\end{equation}
\begin{equation}
\begin{split}
\frac{1}{c_{l,l+1}}\left( \mathbf{I}_{l}^{l+1}(\hat{\mathbf{x}})\right) ^{(2)}|\mathbf{x}|^l
=&\mathtt{i}(l+1)A_l(\sin\theta)^{l} e^{\mathtt{i}l\varphi}r^{l}\\
=&\mathtt{i}\frac{1}{c_{l,l+1}}\left( \mathbf{I}_{l}^{l+1}(\hat{\mathbf{x}})\right) ^{(1)}|\mathbf{x}|^l,
\end{split}
\end{equation}
and
\begin{equation}
\begin{split}
\frac{1}{c_{l,l+1}}\left( \mathbf{I}_{l}^{l+1}(\hat{\mathbf{x}})\right) ^{(3)}|\mathbf{x}|^l
=&0.
\end{split}
\end{equation}
Hence the factors $\mathbf{x}^{\bm{\alpha}}$ of the form \eqref{eq:factor} in $\mathbf{I}_{l}^{m}(\hat{\mathbf{x}})|x|^l$ can be summarized into the following table:
\begin{equation*}
\begin{array}{c}
\qquad\begin{array}{ccl} 	
&  \mbox{$m=(l+1)$ is even}
&  \mbox{$m=(l+1)$ is odd}
\\ \hline
& \qquad\begin{array}{ll}
\Re & \qquad \Im
\end{array}
& \begin{array}{rr}
\Re &  \qquad\Im
\end{array}
\end{array}
\\
\begin{array}{c|l|l|l|l}\hline
\left( \mathbf{I}_{l}^{m}(\hat{\mathbf{x}})\right) ^{(1)}|x|^l
& x^{(1)} & x^{(2)}& \text{None}& x^{(1)}x^{(2)}
\\
\left( \mathbf{I}_{l}^{m}(\hat{\mathbf{x}})\right) ^{(2)}|x|^l
& x^{(2)} & x^{(1)}
& x^{(1)}x^{(2)}	& \text{None}
\end{array}
\end{array}
\end{equation*}
Therefore, for $m=l+1$, there is never the case \eqref{eq:Peven}, and the validity for \eqref{eq:Podd} is equivalent to that $m=l+1$ is even (and hence $l$ is odd), and the corresponding term is
\begin{equation}\label{eq:eq6.49}
\Re \mathbf{I}_{l}^{l+1}(\mathbf{x})= \frac{\mathtt{1}}{2}\left( \mathbf{I}_{l}^{l+1}(\mathbf{x})+\mathbf{I}_{l}^{-(l+1)}(\mathbf{x})\right).
\end{equation}

Next we assume that $0\le m\le l$.
It is derived from \eqref{eq:Plm} that
\begin{equation}
\begin{split}
P_{l+1}^{m}(\cos\theta)
=&\sum_{n=0}^{l'}\frac{p_{l+1}^{(2n)}(l-2n+1)!}{(2l'-2n)!}(\cos\theta)^{2l'-2n}(\sin\theta)^{m}\\
=&\sum_{n=0}^{l'}p_{l+1}^{(2l'-2n)}\frac{(m+2n)!}{(2n)!}(\cos\theta)^{2n}(\sin\theta)^{m}\\
=&\sum_{n=0}^{l'}(-1)^n B_{n,m,l'} (\sin\theta)^{m+2n},
\end{split}	
\end{equation}
with the coefficients
\begin{equation}
B_{n,m,l'}
:=\sum_{t=n}^{l'} \frac{(m+2t)!}{(2t)!}
\complement_{t}^{n}p_{l+1}^{(2l'-2t)},
\end{equation}
and hence that
\begin{equation}
\begin{split}
\frac{dP_{l+1}^{m}(\cos{\theta})}{d\theta}
=&\sum_{n=0}^{l'}(-1)^n(m+2n)B_{n,m,l'}\cos\theta (\sin\theta)^{m+2n-1}.
\end{split}	
\end{equation}	
Recalling the expression \eqref{eq:Ilm_r}, one can obtain that
\begin{equation}\label{eq:eq2.75}
\begin{split}
&\frac{1}{c_{l,m}}\left( \mathbf{I}_{l}^{m}(\hat{\mathbf{x}})\right) ^{(1)}|\mathbf{x}|^l\\
=&(l+1)r^{l}\sum_{n=0}^{l'}(-1)^nB_{n,m,l'} (\sin\theta)^{m+2n+1}\cos\varphi e^{\mathtt{i}m\varphi}\\
&+r^{l} (\cos\theta)^2\sum_{n=0}^{l'}(m+2n)(-1)^nB_{n,m,l'} (\sin\theta)^{m+2n-1}e^{\mathtt{i}m\varphi}\cos\varphi\\
&-\mathtt{i}r^{l}m\sum_{n=0}^{l'}(-1)^nB_{n,m,l'} (\sin\theta)^{m+2n-1}\sin\varphi e^{\mathtt{i}m\varphi}.\\
\end{split}
\end{equation}
Introducing the coefficients $\tilde{B}_{n,m,l'}$ by
\begin{equation}
\begin{split}
\tilde{B}_{n,m,l'}
:=&nB_{n,m,l'}-(l'-n+1)B_{n-1,m,l'}\\
=&\sum_{t=n}^{l'}(l'+t-2n+2) p_{l+1}^{(2l'-2t)}\frac{(m+2t)!}{(2t)!}
\complement_{t}^{n-1}\\
&+p_{l+1}^{(2l'-2n+1)}\frac{(m+2n-2)!}{(2n-2)!},
\end{split}
\end{equation}
the formula \eqref{eq:eq2.75} can be further reduced to
\begin{equation}
\begin{split}
\frac{1}{c_{l,m}}\left( \mathbf{I}_{l}^{m}(\hat{\mathbf{x}})\right) ^{(1)}|\mathbf{x}|^l
=&2\sum_{n=1}^{l'}(-1)^n\tilde{B}_{n,m,l'} (\sin\theta)^{m+2n-1}\cos\varphi e^{\mathtt{i}m\varphi}r^{l}\\
&+ m\sum_{n=0}^{l'}(-1)^nB_{n,m,l'} (\sin\theta)^{m+2n-1} e^{\mathtt{i}(m-1)\varphi}r^{l}.
\end{split}
\end{equation}
Therefore,
\begin{equation}\label{eq:eq2.81}
\begin{split}
&\frac{1}{c_{l,m}}\left( \mathbf{I}_{l}^{m}(\hat{\mathbf{x}})\right) ^{(1)}|\mathbf{x}|^l\\
=&2\sum_{n=1}^{l'}(-1)^n \tilde{B}_{n,m,l'} x^{(1)}r_3^{n-1}\left(x^{(1)}+\mathtt{i}x^{(2)} \right)^m r^{2l'-2n}\\
&+ m\sum_{n=0}^{l'}(-1)^nB_{n,m,l'}r_3^{n} \left(x^{(1)}+\mathtt{i}x^{(2)} \right)^{m-1}r^{2l'-2n}.
\end{split}
\end{equation}	
Analogously, one has
\begin{equation}
\begin{split}
\frac{1}{c_{l,m}}\left( \mathbf{I}_{l}^{m}(\hat{\mathbf{x}})\right) ^{(2)}|\mathbf{x}|^l
=&2\sum_{n=1}^{l'}(-1)^n \tilde{B}_{n,m,l'}(\sin\theta)^{m+2n-1}\sin\varphi e^{\mathtt{i}m\varphi}r^{l}\\
&+ \mathtt{i}m\sum_{n=0}^{l'}(-1)^n B_{n,m,l'} (\sin\theta)^{m+2n-1} e^{\mathtt{i}(m-1)\varphi}r^{l},
\end{split}
\end{equation}
and hence,	
\begin{equation}\label{eq:eq2.83}
\begin{split}
&\frac{1}{c_{l,m}}\left( \mathbf{I}_{l}^{m}(\hat{\mathbf{x}})\right) ^{(2)}|\mathbf{x}|^l\\
=&2\sum_{n=1}^{l'}(-1)^n \tilde{B}_{n,m,l'} x^{(2)}r_3^{n-1}\left(x^{(1)}+\mathtt{i}x^{(2)} \right)^m r^{2l'-2n}\\
&+\mathtt{i} m\sum_{n=0}^{l'}(-1)^nB_{n,m,l'}r_3^{n} \left(x^{(1)}+\mathtt{i}x^{(2)} \right)^{m-1}r^{2l'-2n}.
\end{split}
\end{equation}
Moreover,
\begin{equation}\label{eq:eq2.84}
\begin{split}
&\frac{1}{c_{l,m}}\left( \mathbf{I}_{l}^{m}(\hat{\mathbf{x}})\right) ^{(3)}|\mathbf{x}|^l\\
=& \sum_{n=0}^{l'}(-1)^n(2l'-2n)B_{n,m,l'} (\sin\theta)^{m+2n}\cos\theta e^{\mathtt{i}m\varphi}r^{l}\\
=& \sum_{n=0}^{l'}(-1)^n(2l'-2n)B_{n,m,l'} x^{(3)}r_3^{n}\left(x^{(1)}+\mathtt{i}x^{(2)} \right)^{m}r^{2l'-2n}.
\end{split}
\end{equation}
Similar to the case when $(l-m+1)$ is odd, one can deduce from \eqref{eq:eq2.81}, \eqref{eq:eq2.83} and \eqref{eq:eq2.84}
that, for $0\le m\le l$,
\begin{equation*}
\begin{array}{c}
\qquad\begin{array}{ccc} 	
&  \qquad \mbox{$m$ is even}
&  \qquad \mbox{$m$ is odd}
\\ \hline
& \qquad\begin{array}{ll}
\Re & \qquad \Im
\end{array}
& \qquad\begin{array}{rr}
\Re &  \qquad\Im
\end{array}
\end{array}
\\
\begin{array}{c|l|l|l|l}\hline
\left( \mathbf{I}_{l}^{m}(\hat{\mathbf{x}})\right) ^{(1)}|x|^l
& x^{(1)} & x^{(2)}& \text{None}& x^{(1)}x^{(2)}
\\
\left( \mathbf{I}_{l}^{m}(\hat{\mathbf{x}})\right) ^{(2)}|x|^l
& x^{(2)} & x^{(1)}& x^{(1)}x^{(2)}	& \text{None}
\\
\left( \mathbf{I}_{l}^{m}(\hat{\mathbf{x}})\right) ^{(3)}|x|^l
& x^{(3)} & x^{(1)}x^{(2)}x^{(3)} & x^{(1)}x^{(3)} & x^{(2)}x^{(3)}
\end{array}
\end{array}
\end{equation*}

Recalling \eqref{eq:eq6.49}, one can see that for $(l+m+1)$ even and $0\le m\le l+1$, the case \eqref{eq:Peven} cannot occur, and the validity of \eqref{eq:Podd} is equivalent to that $m$ is even (and hence $l$ is odd), with the corresponding term given by
\begin{equation}\label{eq:Even}
\Re \mathbf{I}_{l}^{m}(\mathbf{x})= \frac{\mathtt{1}}{2}\left( \mathbf{I}_{l}^{m}(\mathbf{x})+\mathbf{I}_{l}^{-m}(\mathbf{x})\right),
\quad \mbox{$l$ is odd, $m$ is even, $0\le m\le l+1$}.
\end{equation} 	

Finally, we can conclude the proof of Lemma~\ref{lem:a1} by combining \eqref{eq:Odd} and \eqref{eq:Even}.

\section*{Acknowledgement}

The work was supported by the FRG fund from Hong Kong Baptist University, the Hong Kong RGC grants (projects 12302415) and NSF of China (No.\,11371115).

\addcontentsline{toc}{section}{Bibliography}
\bibliographystyle{siam}
\bibliography{all}

\end{document}